\documentclass[11pt,reqno]{amsart}
\usepackage[T1]{fontenc}
\usepackage[utf8]{inputenc}
\usepackage{lmodern}
\usepackage{amsmath,amssymb,mathtools,bm}
\usepackage[a4paper,margin=28mm]{geometry}
\usepackage{microtype,graphicx}
\usepackage{enumitem,needspace,placeins,float}
\usepackage[hidelinks,pdfencoding=auto]{hyperref}
\numberwithin{equation}{section}
\newtheorem{theorem}{Theorem}[section]
\newtheorem{lemma}[theorem]{Lemma}

\newtheorem{corollary}[theorem]{Corollary}
\theoremstyle{definition}
\newtheorem{assumption}[theorem]{Assumption}
\theoremstyle{remark}
\newtheorem{remark}[theorem]{Remark}
\newcommand{\R}{\mathbb R}
\newcommand{\E}{\mathbb E}
\newcommand{\Prob}{\mathbb P}
\newcommand{\F}{\mathcal F}
\newcommand{\thh}{\widehat\theta}
\newcommand{\bhat}{\widehat b}
\newcommand{\ahat}{\widehat a}
\newcommand{\ph}{\varphi}
\newcommand{\ups}{\upsilon}

\newcommand{\diag}{\operatorname{diag}}
\newcommand{\as}{\quad\text{a.s.}}
\newcommand{\norm}[1]{\lVert#1\rVert}

\newcommand{\calB}{\mathcal B}
\newcommand{\ellfun}{\mathsf L}
\setlist[enumerate]{label=(\roman*),leftmargin=*,itemsep=2pt,topsep=4pt}
\allowdisplaybreaks[1]
\title[Control and Identification in Self-Tuning Regulation]{Control and Identification in Least-Squares Self-Tuning Regulation with Unknown Input Gain}
\author{Zhaobo Liu}
\address{Institute for Advanced Study, Shenzhen University, Shenzhen 518060, China}
\email{liuzhaobo@szu.edu.cn}
\thanks{This work was supported in part by the National Natural Science Foundation of China (12401585), the Guangdong Basic and Applied Basic Research Foundation (2024A1515011542), and the Shenzhen Science and Technology Program through the General Program of the Shenzhen Natural Science Foundation (JCYJ20250604181037012).}
\hypersetup{pdftitle={Control and Identification in Least-Squares Self-Tuning Regulation with Unknown Input Gain},pdfauthor={Zhaobo Liu}}
\date{}
\begin{document}
\begin{abstract}
We study stability, minimum variance tracking and parameter identification for the {\AA}str{\"o}m--Wittenmark self-tuning regulator with unknown input gain. Ordinary least squares (LS) estimates are used directly in certainty-equivalent control. Under the stated plant and martingale noise assumptions, we establish almost sure stability and optimal tracking for every bounded reference, provided that the gain estimates remain nonzero. When at least two parameters are estimated, all coefficient estimates are strongly consistent. The proof obtains a logarithmic information lower bound from an LS energy identity and uses it to establish gain consistency without assuming closed-loop stability. For at least two estimated parameters and every bounded reference, the squared error is almost surely $O(\log\log n/n)$ for the estimated ratios of the plant coefficients to the input gain and $O(\log\log\log n/\log n)$ for the full parameter estimate. A law of the iterated logarithm shows that the latter rate is sharp under zero reference. The bound for the full parameter estimate improves with reference energy, which also appears explicitly in the logarithm law for cumulative squared tracking residuals. These results show that optimal tracking can coexist with substantially slower estimation of the full parameter vector than of the coefficient ratios used in feedback. For at least two estimated parameters, a chi-square limit yields asymptotically exact confidence ellipsoids for every bounded reference.
\end{abstract}
\maketitle
\section{Introduction}\label{sec:introduction}
Adaptive control uses data gathered during operation to stabilize systems with unknown coefficients and achieve desired performance~\cite{TZMP23,ZS25}. A classical example is the self-tuning regulator (STR) of {\AA}str{\"o}m and Wittenmark~\cite{AW73}, which estimates the coefficients by least squares (LS) and uses the estimates in certainty-equivalent (CE) control. The resulting input affects subsequent observations used to update the estimates~\cite{Ljung77,LW82,CG86}. Analyzing the STR therefore involves determining both its tracking performance and what these feedback data reveal about the unknown coefficients.

We consider a single-input single-output system
\begin{equation}\label{eq:plant-model}
 A(z)y_{n+1}=B(z)u_n+w_{n+1},\qquad n\ge0.
\end{equation}
Here $y_n$, $u_n$ and $w_n$ denote the output, input and process noise, and $z$ is the backward shift. The plant polynomials are
\begin{equation*}
 A(z)\triangleq1+\sum_{i=1}^p a_i z^i,\qquad
 B(z)\triangleq\sum_{j=1}^q b_j z^{j-1}.
\end{equation*}
Here $p\ge0$ and $q\ge1$ are fixed order bounds. All coefficients, including the nonzero input gain $b_1$, are unknown. For bounded references, we aim to keep average input and output energy bounded and minimize the asymptotic average squared tracking error. With known coefficients, the input can be chosen so that the tracking error equals the next noise sample. Minimum variance tracking means that the average squared tracking error converges almost surely to the noise variance~\cite[Section~1.2]{Guo94}.

The original STR uses ordinary LS~\cite{AW73} to estimate plant coefficients. Subsequent studies use stochastic gradient~\cite{GRC80}, modified LS~\cite{SG82,Chen84} and weighted LS~\cite{Bercu95}. Recursive estimators with nonlinear transformations of the prediction error have also been studied for multivariable self-tuning tracking~\cite{Filipovic99}. Beyond the choice of estimator, controller design can exploit prior knowledge or introduce excitation. Knowing the gain's sign and a positive lower bound on its magnitude supports robustness to model perturbations~\cite{PLK89}. Probing provides identification information through an additional input without relying on reference excitation~\cite{LW86,CG87,LW87,BV10PE} and can yield consistency even when the probing signal decays at a suitable rate~\cite[Theorem~7.3]{Guo95}. For adaptive linear quadratic control from online closed-loop data, Zhao et al.~\cite{ZDCY25} bound the average cost gap under persistent input excitation and boundedness assumptions. Lale et al.~\cite{Lale21} give finite-time guarantees for estimation and control of autoregressive systems with exogenous inputs under sufficient excitation. For STR, Singh et al.~\cite{Singh24} obtain logarithmic regret under bounded noise using probing and input clipping.

Guo and Chen~\cite[Theorem~1]{GC91} showed that CE control with ordinary LS and known gain stabilizes minimum-phase plants and tracks bounded references optimally under suitable noise conditions. Guo~\cite[Theorem~4.1]{Guo95} further quantified tracking performance through bounds on the residuals, namely tracking errors minus noise. When $A-1$ and $B$ are coprime and at least one polynomial has degree equal to its upper bound, Guo obtained strong consistency and a convergence rate for the parameter estimates, together with an exact logarithm law for cumulative squared tracking residuals~\cite[Theorem~5.1]{Guo95}. These results require no reference excitation. With unknown gain, the controller divides by the gain estimate. Small estimates can therefore produce large inputs, and zero estimates make the recursion undefined~\cite{MC85,Kumar90,Guo95}. By adjusting the gain estimate, Guo and Chen~\cite[Theorem~2]{GC91} obtained stability and optimal tracking without coprimeness (see also~\cite[Theorem~6.3]{Guo95}). Zhang and Guo~\cite[Theorem~3.1]{ZG21} combined weighted LS with a modification of the estimate used in feedback to establish stability and asymptotically optimal tracking under conditionally heteroscedastic noise.

For direct LS feedback with unknown gain, Kumar~\cite[Theorem~8(i)]{Kumar90} used Bayesian embedding under i.i.d.\ Gaussian noise to prove optimal tracking of bounded deterministic references for plant parameters outside a Lebesgue null set. With one-step delay and the coprimeness and degree conditions above, the estimates converge to a random multiple of the true parameter vector~\cite[Theorem~8(ii)]{Kumar90}. Strong consistency follows if the mean square of the reference filtered by $B$ has a positive limit superior~\cite[Theorem~8(iii)]{Kumar90}. Guo~\cite[Theorem~4.2]{Guo94}, \cite[Theorem~7.2]{Guo95} proved stability, optimal tracking and consistency with rates for coprime $A,B$ under sufficient information growth from lagged references. With at least two estimated parameters, this condition requires both growing reference energy and variation among lagged reference values. It therefore excludes zero and constant references.

Without reference excitation, the analysis faces two related difficulties. The input depends on the reciprocal of the gain estimate, so deriving gain convergence from stable closed-loop data would assume part of the desired conclusion. Under zero reference, multiplying $A-1$ and $B$ by the same nonzero constant leaves the ideal minimum variance input and output responses to noise unchanged. These responses leave the common parameter scale, and hence the input gain, undetermined. The issue is how the adaptive recursion accumulates information about this scale while stability is still to be established. For every bounded reference, can ordinary LS used directly in feedback achieve stability and optimal tracking, and when does it yield strong consistency?

We show that ordinary LS used directly in CE control achieves stability and optimal tracking almost surely for every bounded reference. This holds for minimum-phase plants with $A-1$ and $B$ coprime under the degree and noise conditions in Section~\ref{sec:model}, provided that the gain estimates remain nonzero. Under the same conditions, the parameter estimates are strongly consistent when at least two parameters are estimated.

To prove gain consistency in this case, we consider subsequences on which the gain estimation error is bounded away from zero. An LS energy identity compares fits to the noise alone and to the noise plus the reference with the same regressors. The comparison yields a logarithmic information lower bound for the common scale of the plant coefficients. This bound, Guo's estimates~\cite{Guo94,Guo95} and the LS normal equation rule out such subsequences.

The ratios of plant coefficients to the input gain determine how past inputs and outputs enter the control law. With at least two estimated parameters, the squared errors of the ratio estimates and of the full parameter estimate are almost surely $O(\log\log n/n)$ and $O(\log\log\log n/\log n)$, respectively, for every bounded reference. The latter bound is sharp under zero reference. The squared error bound for the full parameter estimate improves with cumulative reference energy. For a nonzero constant reference, the squared error bounds for the full parameter estimate and the ratio estimates are both $O(\log\log n/n)$, although Guo's reference condition is not satisfied. Cumulative reference energy also appears explicitly in the logarithm law for cumulative squared tracking residuals.

When at least two parameters are estimated, we also derive their asymptotic normal law. A chi-square limit gives confidence ellipsoids with asymptotically exact coverage for every bounded reference.

The remainder of the paper is organized as follows.
Sections~\ref{sec:model} and~\ref{sec:main} present the problem formulation and main results.
Sections~\ref{sec:identification}--\ref{sec:cost} establish stability, optimal tracking and consistency, derive convergence rates, and prove the logarithm law and sharpness, respectively.
Section~\ref{sec:distribution-proof} proves the distributional results, and Section~\ref{sec:numerical} presents numerical illustrations.

\subsection{Notation}
Vector norms are Euclidean and matrix norms are induced, except for the Frobenius norm $\|\cdot\|_F$. For symmetric matrices, $H>0$ means positive definite and $H\succeq K$ means $H-K$ is positive semidefinite. Write $\|v\|_H^2\triangleq v^THv$ for $H>0$. The backward shift satisfies $zv_n=v_{n-1}$. Write $\Delta v_n\triangleq v_{n+1}-v_n$ for the forward difference. For a causal filter $F(z)$, both $[F(z)v]_n$ and $(F(z)v)_n$ denote its output at time $n$, with zero initial history unless specified otherwise. Unless otherwise indicated, $\sum_{t<n}$ means $\sum_{t=0}^{n-1}$.

Asymptotic relations concern $n\to\infty$. For nonnegative scalars, $X\lesssim Y$ or $Y\gtrsim X$ means $X\le CY$ for some $C>0$, eventually for sequences. For positive semidefinite matrices, $X\lesssim Y$ means $X\preceq CY$. Write $X\asymp Y$ when $X\lesssim Y$ and $Y\lesssim X$. The symbols $O$ and $o$ refer to absolute value or norm. Constants may vary between estimates but are independent of time indices and test vectors. In pathwise estimates, they may depend on the sample path.

\section{Problem Formulation}\label{sec:model}
\subsection{Model and regulator}
The model~\eqref{eq:plant-model} has the regression form $y_{n+1}=\theta^T\ph_n+w_{n+1}$, where $d\triangleq p+q$ and
\[
\begin{gathered}
\theta\triangleq(-a_1,\ldots,-a_p,b_1,\ldots,b_q)^T\in\R^d,\\
 \ph_n\triangleq(y_n,\ldots,y_{n-p+1},u_n,\ldots,u_{n-q+1})^T.
\end{gathered}
\]
The output block of $\ph_n$ is empty when $p=0$. Let $e_{p+1}\in\R^d$ be the $(p+1)$st standard basis vector.

All initial input and output values, including $y_0$, are fixed. Extend $w_n$ by zero for $n\le0$. Let $y_n^*$, $n\ge1$, denote the reference signal, and set $y_n^*=0$ for $n\le0$.

The recursive update for ordinary LS is
\begin{align}
 \thh_{n+1}&=\thh_n+\frac{P_n\ph_n}{1+\ph_n^TP_n\ph_n}
       (y_{n+1}-\thh_n^T\ph_n),\label{eq:rls}\\
 P_{n+1}^{-1}&=P_n^{-1}+\ph_n\ph_n^T,\qquad P_0>0.\label{eq:info-update}
\end{align}
The initial estimate $\thh_0$ and matrix $P_0$ are fixed. The CE control equation is
\begin{equation}\label{eq:ce}
 \thh_n^T\ph_n=y_{n+1}^*.
\end{equation}
With $\thh_n=(-\ahat_{1,n},\ldots,-\ahat_{p,n},\bhat_{1,n},\ldots,\bhat_{q,n})^T$, the control law \eqref{eq:ce} takes the form
\begin{equation*}
 u_n=\frac{y_{n+1}^*+\sum_{i=1}^p\ahat_{i,n}y_{n-i+1}
                   -\sum_{j=2}^q\bhat_{j,n}u_{n-j+1}}{\bhat_{1,n}}.
\end{equation*}
At each $n\ge0$ with $\bhat_{1,n}\ne0$, the controller selects $u_n$ using $\thh_n$ and the known reference $y_{n+1}^*$, then updates $\thh_{n+1}$ after observing $y_{n+1}$.

\subsection{Assumptions}\label{sec:assumptions}
Define the filtration
$\F_n\triangleq\sigma(y_j^*:j\ge1;w_1,\ldots,w_n)$.
\begin{assumption}\label{ass:noise}
The reference satisfies $\sup_{n\ge1}|y_n^*|<\infty$ almost surely, and $\ph_n$ is $\F_n$-measurable for every $n\ge0$. For some constants $\sigma^2>0$ and $\nu>2$, the following hold almost surely:
\begin{equation}\label{eq:noise}
\begin{gathered}
\E[w_{n+1}\mid\F_n]=0,\quad \E[w_{n+1}^2\mid\F_n]=\sigma^2,\quad n\ge0,\\
\sup_{n\ge0}\E[|w_{n+1}|^\nu\mid\F_n]<\infty.
\end{gathered}
\end{equation}
\end{assumption}
The noise conditions follow the martingale difference framework of Guo~\cite[Condition~A.1]{Guo95}. Here they are imposed with respect to the filtration containing the complete reference path.

\begin{assumption}\label{ass:plant}
The polynomial $B(z)$ has no zeros in $|z|\le1$, $\gcd(A-1,B)=1$, and at least one of $\deg(A-1)=p$ and $\deg B=q-1$ holds, with $\deg0=-\infty$.
\end{assumption}
The coprimeness of $A-1$ and $B$, together with the degree condition, also appears in Kumar's parameter convergence results~\cite[Theorems 5(iii) and 6(iii)]{Kumar90} and Guo's consistency and rate result for known input gain~\cite[Theorem 5.1]{Guo95}. For unknown gain, Guo~\cite[Theorem~7.2]{Guo95} instead assumes $A$ and $B$ coprime. Neither coprimeness condition implies the other.

As in~\cite[eq. (70)]{Guo95}, we assume that
\begin{equation}\label{eq:well-defined}
 \Prob\{\bhat_{1,n}\ne0\text{ for every integer }n\ge0\}=1.
\end{equation}
Condition~\eqref{eq:well-defined} holds if $\bhat_{1,0}\ne0$ and, for every $n\ge0$, the conditional law of $w_{n+1}$ given $\F_n$ has no atoms almost surely. With bounded discrete noise, a gain estimate can vanish with positive probability. Appendix~\ref{app:boundary} gives the proof and a counterexample.

\subsection{Control and identification objectives}\label{sec:objectives}
Define $\pi_n$ and the tracking residual $\varepsilon_n$ by
$\pi_n\triangleq\theta^T\ph_n=y_{n+1}-w_{n+1},\ \varepsilon_n\triangleq\pi_n-y_{n+1}^*,$
and the cumulative residual, regressor and reference energies by
\[
\begin{gathered}
R_n\triangleq\sum_{j=0}^{n-1}\varepsilon_j^2,\quad
 r_n\triangleq1+\sum_{j=0}^{n-1}\norm{\ph_j}^2,\quad
 E_n\triangleq\sum_{j=0}^{n-1}(y_{j+1}^*)^2.
\end{gathered}
\]
The tracking error satisfies $y_{n+1}-y_{n+1}^*=\varepsilon_n+w_{n+1}$.

Following~\cite[Section 1.2]{Guo94}, stability means bounded average input and output energy,
\begin{equation}\label{eq:stability}
 \limsup_{n\to\infty}\frac1n\sum_{j=0}^{n-1}(y_j^2+u_j^2)<\infty\as
\end{equation}

Since $\varepsilon_n$ is $\F_n$-measurable, Assumption~\ref{ass:noise} implies
\[
 \E[(y_{n+1}-y_{n+1}^*)^2\mid\F_n]=\varepsilon_n^2+\sigma^2.
\]
With known coefficients and $b_1\ne0$, choosing $u_n$ by $\theta^T\ph_n=y_{n+1}^*$ sets $\varepsilon_n=0$ and attains the minimum conditional mean squared tracking error $\sigma^2$.

Tracking is asymptotically optimal when
\begin{equation}\label{eq:tracking}
 \frac1n\sum_{j=1}^{n}(y_j-y_j^*)^2
 =\frac1n\sum_{j=1}^{n}w_j^2+o(1)\as
\end{equation}

By the strong law for $w_j^2$ under Assumption~\ref{ass:noise}~\cite[Theorem~2.18, pp.~35--36]{HallHeyde80}, \eqref{eq:tracking} is equivalent to
\[
 \frac1n\sum_{j=1}^{n}(y_j-y_j^*)^2\longrightarrow\sigma^2\as
\]
This is also equivalent to $R_n=o(n)$ almost surely~\cite[p.~437]{Guo95}.

The identification objective is strong consistency of the parameter estimates, that is, $\thh_n\to\theta$ almost surely.

\section{Main Results}\label{sec:main}

\subsection{Stability, tracking and parameter convergence}

\begin{theorem}\label{thm:basic}
Under Assumptions~\ref{ass:noise}--\ref{ass:plant} and \eqref{eq:well-defined}, the recursion \eqref{eq:rls}--\eqref{eq:ce} satisfies
\begin{equation*}
 r_n=O(n),\qquad R_n=O(\log n)\as
\end{equation*}
Consequently, the closed-loop system is stable in the sense of \eqref{eq:stability} and achieves optimal tracking as defined in \eqref{eq:tracking}. If $d\ge2$, the parameter estimates are strongly consistent.
\end{theorem}
The proof is given in Sections~\ref{sec:identification}--\ref{sec:rates}.
\begin{remark}\label{rem:d1}
If $p=0$, $q=1$ and $y_n^*=0$, then $u_n=0$ and $\bhat_{1,n}=\bhat_{1,0}$ for all $n$ whenever $\bhat_{1,0}\ne0$. Thus consistency can fail for $d=1$, although optimal tracking holds.
\end{remark}

To describe how reference energy affects the convergence rates when $d\ge2$, we separate the estimation error along the true parameter vector from the component orthogonal to it. Let the columns of $Q\in\R^{d\times(d-1)}$ form an orthonormal basis of $\theta^\perp$. Write
\begin{equation}\label{eq:coords}
 \thh_n=s_n\theta+Q\ups_n.
\end{equation}
Thus $(s_n-1)\theta$ is the component of the estimation error parallel to $\theta$, and $Q\ups_n$ is the component orthogonal to $\theta$.

\begin{theorem}\label{thm:rates}
Under the conditions of Theorem~\ref{thm:basic}, with $d\ge2$, the following bounds hold almost surely.
\begin{equation}\label{eq:rate-main}
 (s_n-1)^2=O\!\left(\frac{\log\log(E_n+\log n)}{E_n+\log n}\right),\qquad
 \norm{\ups_n}^2=O\!\left(\frac{\log\log n}{n}\right).
\end{equation}
Consequently,
\begin{equation}\label{eq:full-rate}
 \norm{\thh_n-\theta}^2
 =O\!\left(\frac{\log\log(E_n+\log n)}{E_n+\log n}\right).
\end{equation}
\end{theorem}
\begin{remark}
Reference energy enters the bound for $(s_n-1)^2$, while the orthogonal component has squared norm $O(\log\log n/n)$ for every bounded reference. The bound in~\eqref{eq:full-rate} also implies
\begin{equation}\label{eq:uniform-rate}
 \norm{\thh_n-\theta}^2
 =O\!\left(\frac{\log\log\log n}{\log n}\right)\as
\end{equation}
When $E_n\asymp n$, both bounds in~\eqref{eq:rate-main} are $O(\log\log n/n)$ almost surely.
Under a nonzero constant reference, $E_n\asymp n$ and the lagged reference vectors are eventually identical, so the smallest eigenvalue of the sum of their outer products remains bounded for $d\ge2$ and Guo's reference condition in~\cite[Theorem~7.2]{Guo95} fails.

\end{remark}

\begin{corollary}\label{cor:ratios}
Under the conditions of Theorem~\ref{thm:basic}, with $d\ge2$,
\begin{equation*}
 \norm{\frac{\thh_n}{\bhat_{1,n}}-\frac\theta{b_1}}^2
 =O\!\left(\frac{\log\log n}{n}\right)\as
\end{equation*}
\end{corollary}
Theorem~\ref{thm:rates} and Corollary~\ref{cor:ratios} are proved in Section~\ref{sec:rates}.
\begin{remark}
Ratios of the coefficient estimates to the estimated input gain are invariant under common nonzero rescaling and determine how past outputs and inputs enter the control law. When $y_{n+1}^*\ne0$, the reference term $y_{n+1}^*/\bhat_{1,n}$ also depends on $s_n$, whose error bound in Theorem~\ref{thm:rates} improves with reference energy.
\end{remark}

\subsection{Logarithm laws and sharpness}
We now refine the bound $R_n=O(\log n)$ by relating cumulative squared tracking residuals to the LS information matrix. Write
\begin{equation*}
 \ell_n\triangleq\log\frac{\det(P_n^{-1})}{\det(P_0^{-1})}.
\end{equation*}

\begin{theorem}\label{thm:cost}
Under the conditions of Theorem~\ref{thm:basic}, with $d\ge2$, the following conclusions hold almost surely.
\begin{align}
 R_n&=\sigma^2\ell_n+
 O\!\left(\sqrt{\log n\,\log\log\log n}\right),\notag\\
 &=\sigma^2\{(d-1)\log n+\log(E_n+\log n)\}
 +O\!\left(\sqrt{\log n\,\log\log\log n}\right)\label{eq:R-ell}
\end{align}
In particular $R_n/(\sigma^2\ell_n)\to1$.
\end{theorem}

\begin{remark}
For known input gain, with $A-1$ and $B$ coprime and the degree condition, Guo~\cite[Theorem~5.1]{Guo95} established a logarithm law with coefficient $(d-1)\sigma^2$. For unknown gain, let $Y_j^*\triangleq(y_j^*,\ldots,y_{j-d+1}^*)^T$. Under Guo's plant and noise assumptions, including coprimeness of $A$ and $B$, the corresponding coefficient is $d\sigma^2$ if $n=O(\lambda_{\min}(\sum_{j=1}^nY_j^*(Y_j^*)^T))$ almost surely~\cite[Theorem~4.2]{Guo94}, \cite[Theorem~7.2(ii)]{Guo95}. In Theorem~\ref{thm:cost}, the coefficient for plants satisfying Assumption~\ref{ass:plant} is $d\sigma^2$ when $E_n\asymp n$ and $(d-1)\sigma^2$ under zero reference.
\end{remark}

\begin{corollary}\label{cor:sharp}
Under the conditions of Theorem~\ref{thm:basic}, with $d\ge2$, suppose $y_n^*=0$ for every $n$. The following relations hold almost surely.
\begin{equation}\label{eq:zero-info}
 R_n\sim(d-1)\sigma^2\log n,
\end{equation}
and
\begin{align}
 \limsup_{n\to\infty}\frac{\log n}{\log\log\log n}(s_n-1)^2
 &=\frac{2}{d-1},\label{eq:scale-sharp}\\
 \limsup_{n\to\infty}\frac{\log n}{\log\log\log n}
       \norm{\thh_n-\theta}^2
 &=\frac{2\norm\theta^2}{d-1}\label{eq:vector-sharp}
\end{align}
\end{corollary}
Theorem~\ref{thm:cost} and Corollary~\ref{cor:sharp} are proved in Section~\ref{sec:residual-proof}.
For each nonzero true coefficient $\theta_j$, the limit superior in~\eqref{eq:vector-sharp} is $2\theta_j^2/(d-1)$ almost surely when the squared norm is replaced by $(\thh_n-\theta)_j^2$.

\subsection{Asymptotic normality and confidence ellipsoids}

Let $\bar M_n$ be the comparison matrix defined in~\eqref{eq:distribution-information}, which depends on the reference path, the plant coefficients and the noise variance. Estimate $\sigma^2$ by
\[
 \widehat\sigma_n^2\triangleq\frac1n\sum_{j=1}^{n}(y_j-y_j^*)^2.
\]
\begin{theorem}\label{thm:distribution}
Under the conditions of Theorem~\ref{thm:basic}, with $d\ge2$,
\begin{equation}\label{eq:normal-limit}
 \sigma^{-1}\bar M_n^{1/2}\binom{s_n-1}{\ups_n}
 \xrightarrow{d}\mathcal N_d(0,I_d),
\end{equation}
where $(s_n,\ups_n)$ are defined in~\eqref{eq:coords}. Moreover,
\begin{equation}\label{eq:chisquare}
 \frac{(\thh_n-\theta)^TP_n^{-1}(\thh_n-\theta)}
 {\widehat\sigma_n^2}\xrightarrow{d}\chi_d^2.
\end{equation}
The statistic in~\eqref{eq:chisquare} may be assigned any value when $\widehat\sigma_n^2=0$.
\end{theorem}
The proof is given in Section~\ref{sec:distribution-proof}. For $0<\alpha<1$, let $\chi^2_{d,1-\alpha}$ be the $(1-\alpha)$-quantile of $\chi_d^2$. Then
\[
 \bigl\{\vartheta\in\R^d:
 (\thh_n-\vartheta)^TP_n^{-1}(\thh_n-\vartheta)
 \le\widehat\sigma_n^2\chi^2_{d,1-\alpha}\bigr\}
\]
has coverage tending to $1-\alpha$ for each bounded reference.

\section{Stability, Tracking and Consistency}\label{sec:identification}
Throughout this section, Assumptions~\ref{ass:noise}--\ref{ass:plant} and condition~\eqref{eq:well-defined} hold.
Work first with $d\ge2$ and a bounded deterministic reference. We prove gain consistency, which yields stability and optimal tracking.

\subsection{LS identities and regressor decomposition}\label{sec:filter-representations}
Using the decomposition in~\eqref{eq:coords}, put $L_\theta\triangleq[\theta\ Q]$ and define
\[
 x_n\triangleq Q^T\ph_n,\qquad
 M_n\triangleq L_\theta^TP_n^{-1}L_\theta
 =\begin{pmatrix}\mu_n&g_n^T\\g_n&G_n\end{pmatrix}.
\]
Then $\mu_n=\mu_0+\sum_{t<n}\pi_t^2$ and $G_n=G_0+\sum_{t<n}x_tx_t^T$. Define $T_n$ and $S_n$ by
\[
 \binom{T_n}{S_n}\triangleq
 L_\theta^TP_0^{-1}(\thh_0-\theta)+\sum_{t<n}\binom{\pi_t}{x_t}w_{t+1}.
\]
Put
\[
 k_n\triangleq-G_n^{-1}g_n,\qquad J_n\triangleq\mu_n-k_n^TG_nk_n>0.
\]

\begin{lemma}\label{lem:schur}
The coefficient $s_n$ in~\eqref{eq:coords} satisfies
\begin{equation}\label{eq:T}
 J_n(s_n-1)=T_n+k_n^TS_n.
\end{equation}
With $\xi_n\triangleq\pi_n+k_n^Tx_n$ and $\lambda_n\triangleq x_n^TG_n^{-1}x_n$,
\begin{equation}\label{eq:schur-increment}
 \Delta J_n=\frac{\xi_n^2}{1+\lambda_n}.
\end{equation}
\end{lemma}
\begin{proof}
The LS normal equation is
\[
 M_n\binom{s_n-1}{\ups_n}=\binom{T_n}{S_n}.
\]
Eliminating $\ups_n$ yields~\eqref{eq:T}. For $\alpha\in\R$ and $v\in\R^{d-1}$, complete the square to obtain
\begin{equation}\label{eq:schur-quadratic}
 \binom{\alpha}{v}^TM_n\binom{\alpha}{v}
 =J_n\alpha^2+\norm{v-\alpha k_n}_{G_n}^2.
\end{equation}
For $\alpha=1$,~\eqref{eq:info-update} adds $(\pi_n+v^Tx_n)^2$ to this form. The minimum is $J_{n+1}$ by~\eqref{eq:schur-quadratic} at $n+1$. It is attained at $v=k_n-G_n^{-1}x_n\xi_n/(1+\lambda_n)$ and equals $J_n+\xi_n^2/(1+\lambda_n)$.
\end{proof}

To compare the regressors $\ph_n$ with those generated by the noise $w_n$, first define $u_n^w$ with zero initial values by
\begin{equation*}
 B(z)u_n^w=[A(z)-1]w_{n+1},
\end{equation*}
and form $\ph_n^w$ from output $w_n$ and input $u_n^w$. Then $\theta^T\ph_n^w=0$. Define $z_n\triangleq Q^T\ph_n^w$. The impulse response from $w_n$ to $z_n$ decays exponentially by stability of $B^{-1}$. Extend scalar sequences $v$ by $v_j=0$ for $j<0$.

Define the stable filter $\mathcal T$ componentwise by
\[
 \begin{aligned}
 [(\mathcal Tv)_n]_i&\triangleq v_{n-i},
 &\quad 1\le i\le p,\\
 [(\mathcal Tv)_n]_{p+j}&\triangleq[(A/B)v]_{n-j+1},
 &\quad 1\le j\le q.
 \end{aligned}
\]
Then $\theta^T(\mathcal Tv)_n=v_n$. Set $\ph_n^o\triangleq\ph_n^w+[\mathcal T\{y_{j+1}^*\}_{j\ge0}]_n$.
\begin{lemma}\label{lem:geometry}
The regressor differences satisfy
\begin{equation}\label{eq:noise-difference}
 \sum_{t<n}\norm{x_t-z_t}^2\lesssim\mu_n,\qquad
 \sum_{t<n}\norm{\ph_t-\ph_t^o}^2\lesssim1+R_n.
\end{equation}

\end{lemma}
\begin{proof}
For $n\ge p$, subtracting the plant equations gives $B(z)(u_n-u_n^w)=A(z)\pi_n$. Applying $B^{-1}$, accounting for the initial histories in $\omega_n$, and using $\varepsilon_n=\pi_n-y_{n+1}^*$ yields
\[
 \ph_n=\ph_n^w+(\mathcal T\pi)_n+\omega_n, \qquad \ph_n-\ph_n^o=(\mathcal T\varepsilon)_n+\omega_n.
\]
By stability of $B^{-1}$, $\omega_n$ decays exponentially. By absolute summability of the impulse response of $\mathcal T$,
\[
 \sum_{t<n}\norm{x_t-z_t}^2\lesssim1+\sum_{t<n}\pi_t^2\lesssim\mu_n.
\]
With input $\varepsilon_t$ and the same $\omega_t$, $\sum_{t<n}\norm{\ph_t-\ph_t^o}^2\lesssim1+\sum_{t<n}\varepsilon_t^2=1+R_n$, proving~\eqref{eq:noise-difference}.
\end{proof}

\subsection{Information lower bound}
Write $\ellfun(x)\triangleq\log\log(e^e+x)$ and $H_n\triangleq\ellfun(n)$ for \mbox{$x\ge0$}. Put $S_n^w\triangleq S_0+\sum_{t<n}z_tw_{t+1}$. For LS regression of $w_{t+1}$ on $x_t$ with initial information matrix $G_0$ and estimate $G_0^{-1}S_0$, define $\rho_n$, its prediction $\widehat f_n$ of $w_{n+1}$, and $U_n$ by
\[
 \rho_n\triangleq G_n^{-1}S_n,\qquad
 \widehat f_n\triangleq\rho_n^Tx_n,\qquad
 U_n\triangleq S_n^TG_n^{-1}S_n.
\]
Multiply $g_n(s_n-1)+G_n\ups_n=S_n$ by $G_n^{-1}$ to get \mbox{$\rho_n=\ups_n-k_n(s_n-1)$}. By~\eqref{eq:ce},
\begin{equation}\label{eq:reference-prediction-decomposition}
 y_{n+1}^*-\widehat f_n=\pi_n+(s_n-1)\xi_n.
\end{equation}

\begin{lemma}\label{lem:prediction-window}
Almost surely, along any sequence of indices $n\to\infty$ satisfying $\mu_n+E_n=O(\log n)$,
\[
 J_n\gtrsim\log n.
\]
\end{lemma}
\begin{proof}
Set $a=\lceil n^{1/4}\rceil$, $b=\lfloor\sqrt n\rfloor$. For the matrix bound, use~\eqref{eq:noise-difference}, Lemma~\ref{lem:noise-estimates} and the Cauchy--Schwarz inequality. For the noise sum, apply~\cite[Theorem~1.3.10(i)]{Guo2020} to $x_j-z_j$ using~\eqref{eq:noise-difference}. Almost surely,
\begin{equation}\label{eq:perturbations}
 \Big\|G_t-G_0-\sum_{j<t}z_jz_j^T\Big\|=O(\sqrt{t\mu_t}+\mu_t),\qquad
 \norm{S_t-S_t^w}^2=O(\mu_t\log(e+\mu_t)).
\end{equation}
For $a\le t\le n$, the first bound in~\eqref{eq:perturbations}, divided by $t$, is $O(\sqrt{\mu_n/a}+\mu_n/a)=o(1)$. By Lemma~\ref{lem:noise-estimates} and~\eqref{eq:perturbations},
\[
 U_t\lesssim\frac{\norm{S_t^w}^2+\norm{S_t-S_t^w}^2}{t} =O\!\left(H_n+\frac{\mu_n\log(e+\mu_n)}a\right).
\]
Thus, uniformly for $a\le t\le n$,
\begin{equation}\label{eq:window-preliminaries}
 G_t\asymp tI,\qquad U_t=O(H_n).
\end{equation}
For fixed $2/\nu<\gamma<1$, by Lemma~\ref{lem:noise-estimates} and~\eqref{eq:noise-difference}, $\norm{x_t}^2=O(t^\gamma+\mu_n)$ for $a\le t<n$. By~\eqref{eq:window-preliminaries}, $\max_{a\le t<n}\lambda_t=o(1)$.

By~\eqref{eq:info-update}, $M_b\preceq M_n$. In the quadratic form for $M_b$, write $(1,0)^T=(1,k_n^T)^T-(0,k_n^T)^T$.
\begin{align}
 \mu_b&\le2\binom1{k_n}^{T}M_n\binom1{k_n}+2k_n^TG_bk_n\notag\\
 &=2J_n+2k_n^TG_bk_n
 \le2J_n+O((b/n)\mu_n)=O(J_n).\label{eq:window-mu}
\end{align}
Here $G_b\asymp bI$ and $G_n\asymp nI$ by~\eqref{eq:window-preliminaries}, $k_n^TG_nk_n\le\mu_n$, $(b/n)\mu_n=o(1)$ and $J_n\ge J_0>0$.
For $a\le t<b$, apply~\cite[Theorem~1.3.10(i)]{Guo2020} to $\pi_t$. Then $T_t^2=O(\mu_t\log(e+\mu_t))=O(\mu_bH_n)$ since $\mu_t\le\mu_b=O(\log n)$. Use $(k_t^TS_t)^2\le\mu_t U_t$ and~\eqref{eq:window-preliminaries} in~\eqref{eq:T}.
\[
 J_t^2(s_t-1)^2\le2T_t^2+2\mu_tU_t=O(\mu_bH_n).
\]

For LS with response $w_{t+1}$ and regressor $x_t$,~\cite[Lemma~3.1]{Guo95} gives, along the full sequence,
\begin{equation}\label{eq:prediction-energy-identity}
 U_t+[1+o(1)]\sum_{j<t}\frac{\widehat f_j^2}{1+\lambda_j}
 =\sigma^2\sum_{j<t}\frac{\lambda_j}{1+\lambda_j} +o(\log\det G_t)+O(1).
\end{equation}
By the matrix determinant lemma, $\det G_{j+1}/\det G_j=1+\lambda_j$. Summing $\lambda_j/(1+\lambda_j)\le\log(1+\lambda_j)$ bounds the right-hand sum in~\eqref{eq:prediction-energy-identity} by $\log(\det G_t/\det G_0)$. This is $O(\log n)$ at $t=a,b$ by~\eqref{eq:window-preliminaries}. Since $U_t\ge0$,
\[
 \sum_{j<t}\frac{\widehat f_j^2}{1+\lambda_j}=O(\log n),
 \qquad t=a,b.
\]
Thus the $o(1)$ multiplier of the prediction sum in~\eqref{eq:prediction-energy-identity} contributes $o(\log n)$ at both endpoints.

On $a\le t<b$, use $0\le\log(1+\lambda_t)-\lambda_t/(1+\lambda_t)\le\lambda_t\log(1+\lambda_t)$. The difference of the two sums on this interval is therefore at most
\[
 \max_{a\le t<b}\lambda_t\log(\det G_b/\det G_a)=o(\log n).
\]
 Subtracting~\eqref{eq:prediction-energy-identity} at $a,b$ and using~\eqref{eq:window-preliminaries} yields
\begin{equation}\label{eq:window-prediction-energy}
 \sum_{a\le t<b}\frac{\widehat f_t^2}{1+\lambda_t}
 =\sigma^2\log\frac{\det G_b}{\det G_a}+o(\log n) \sim\frac{(d-1)\sigma^2}{4}\log n.
\end{equation}
Put $\zeta_t\triangleq\sum_{j<t}y_{j+1}^*x_j$ and $W_t\triangleq\norm{\zeta_t}_{G_t^{-1}}^2$. Adding $y_{j+1}^*$ to the response replaces $S_t$ by $S_t+\zeta_t$. In $\zeta_t$, write $x_j=z_j+(x_j-z_j)$. Apply~\eqref{eq:weighted-filter} to the first sum and the Cauchy--Schwarz inequality with~\eqref{eq:noise-difference} to the second. Almost surely along the full sequence,
\begin{equation}\label{eq:reference-score}
 \norm{{\textstyle\sum_{j<t}}y_{j+1}^*x_j}^2=O(E_t(t^\gamma+\mu_t)).
\end{equation}
By~\eqref{eq:window-preliminaries} and~\eqref{eq:reference-score}, uniformly for $a\le t\le b$,
\begin{equation}\label{eq:reference-energy-bound}
 W_t\lesssim\norm{\zeta_t}^2/t
 =O(E_tt^{\gamma-1}+E_t\mu_t/t) =O\!\left((\log n)a^{\gamma-1}+\frac{(\log n)^2}{a}\right)=o(1),
\end{equation}
since $E_t+\mu_t\le E_n+\mu_n=O(\log n)$, $a=\lceil n^{1/4}\rceil$ and $\gamma<1$. By~\eqref{eq:reference-energy-bound}, $H_n\max_{a\le t\le b}W_t=o(1)$.

Put $c_t\triangleq\zeta_{t+1}^TG_{t+1}^{-1}x_t$ and $\bar f_t\triangleq(S_t+\zeta_t)^TG_t^{-1}x_t$. The LS quadratic forms for noise plus reference and for noise alone differ by $2\zeta_t^T\rho_t+W_t$. By~\eqref{eq:info-update},
\begingroup
\medmuskip=2mu\relax
\begin{align}
 \Delta(2\zeta_t^T\rho_t+W_t)
 &=(y_{t+1}^*)^2
 +\frac{\widehat f_t^2-(y_{t+1}^*-\bar f_t)^2}{1+\lambda_t}+2c_tw_{t+1},\notag\\
 \Delta W_t&=(y_{t+1}^*)^2-(1+\lambda_t)(y_{t+1}^*-c_t)^2.\label{eq:reference-energy-updates}
\end{align}
\endgroup
Sum the second identity in~\eqref{eq:reference-energy-updates}, with $W_0=0$. Since $W_t,\lambda_j\ge0$,
\begin{equation}\label{eq:reference-coefficient-energy}
 \begin{gathered}
 \sum_{j<t}(1+\lambda_j)(y_{j+1}^*-c_j)^2=E_t-W_t\le E_t,\\
 \sum_{j<t}c_j^2\le2E_t+2\sum_{j<t}(y_{j+1}^*-c_j)^2\le4E_t.
 \end{gathered}
\end{equation}
The coefficients $c_j$ are $\F_j$-measurable. For $t=a,b$, apply~\cite[Theorem~1.3.10(i)]{Guo2020} with~\eqref{eq:reference-coefficient-energy} to obtain
\[
 \sum_{j<t}c_jw_{j+1}
 =O\!\left(\sqrt{(1+E_t)\log(e+E_t)}\right)
 =o(\log n),
\]
where the last equality uses $E_t\le E_n=O(\log n)$.

By the Cauchy--Schwarz inequality and~\eqref{eq:window-preliminaries}, \eqref{eq:reference-energy-bound},
\[
 |2\zeta_t^T\rho_t+W_t|\le2\sqrt{W_tU_t}+W_t=o(1).
\]
Also, $(\zeta_t^TG_t^{-1}x_t)^2\le W_t\lambda_t$. Using $\lambda_t/(1+\lambda_t)\le\log(\det G_{t+1}/\det G_t)$ with~\eqref{eq:window-preliminaries} and~\eqref{eq:reference-energy-bound},
\begin{equation}\label{eq:reference-prediction-error}
 \sum_{a\le t<b}\frac{(\zeta_t^TG_t^{-1}x_t)^2}{1+\lambda_t}
 \le\max_{a\le t\le b}W_t\sum_{a\le t<b}\frac{\lambda_t}{1+\lambda_t} \le\max_{a\le t\le b}W_t\log\frac{\det G_b}{\det G_a}=o(\log n).
\end{equation}
By~\eqref{eq:reference-prediction-decomposition},
\[
 y_{t+1}^*-\bar f_t=\pi_t+(s_t-1)\xi_t-\zeta_t^TG_t^{-1}x_t.
\]
Sum and rearrange the first identity in~\eqref{eq:reference-energy-updates} over $a\le t<b$. Drop $E_a-E_b\le0$ and use the endpoint and noise bounds. The first comparison uses~\eqref{eq:window-prediction-energy}. The second uses~\eqref{eq:reference-prediction-error} and the decomposition above. By~\eqref{eq:schur-increment}, $\sum_{a\le t<b}\xi_t^2/(1+\lambda_t)=J_b-J_a$. The last two bounds use $J_t^2(s_t-1)^2=O(\mu_bH_n)$ and~\eqref{eq:window-mu}, respectively.
\begin{align}
 \log n&\asymp\sum_{a\le t<b}\frac{\widehat f_t^2}{1+\lambda_t}
 \le\sum_{a\le t<b}\frac{(y_{t+1}^*-\bar f_t)^2}{1+\lambda_t}+o(\log n)\notag\\
 &\lesssim\mu_b+\sum_{a\le t<b}\frac{(s_t-1)^2\xi_t^2}{1+\lambda_t}+o(\log n)\notag\\
 &\lesssim\mu_b+\frac{\mu_bH_n(J_b-J_a)}{J_a^2}+o(\log n)\notag\\
 &\lesssim J_n+\frac{J_n^2H_n}{J_a^2}+o(\log n).\label{eq:information-bootstrap}
\end{align}
Absorb the $o(\log n)$ term in~\eqref{eq:information-bootstrap}. Since $J_a\ge J_0>0$, we obtain $J_n\gtrsim\sqrt{\log n/H_n}$.
Monotonicity of $\mu_t+E_t$ gives $\mu_a+E_a\le\mu_n+E_n=O(\log a)$, so $J_a\gtrsim\sqrt{\log a/H_a}$.
With $\log a\asymp\log n$ and $H_a\asymp H_n$, insert $J_a^2\gtrsim\log n/H_n$ into~\eqref{eq:information-bootstrap} to obtain
\[
 \log n\lesssim J_n+\frac{J_n^2H_n^2}{\log n}.
\]
The larger term on the right is $\gtrsim\log n$, so $J_n\gtrsim\log n/H_n$. Since $\mu_a+E_a=O(\log a)$, the same bound applies at $a$ and gives $J_a^2\gtrsim(\log n/H_n)^2$. With $J_n\le\mu_n=O(\log n)$,
\[
 \frac{J_n^2H_n}{J_a^2}\lesssim\frac{J_n^2H_n^3}{(\log n)^2}=O(H_n^3).
\]
Hence $\log n\lesssim J_n+H_n^3$ by~\eqref{eq:information-bootstrap}. Since $H_n^3=o(\log n)$, we obtain $J_n\gtrsim\log n$.
\end{proof}

\subsection{Gain convergence}\label{sec:gain-information}

Define $V_n\triangleq(\thh_n-\theta)^TP_n^{-1}(\thh_n-\theta)$. By~\eqref{eq:coords} and~\eqref{eq:info-update}, $(s_n-1)^2\lesssim V_n$. Use $\ups_n-k_n(s_n-1)=\rho_n$ in~\eqref{eq:schur-quadratic} to write
\begin{equation}\label{eq:V-decomposition}
 V_n=\binom{s_n-1}{\ups_n}^{\!T}M_n\binom{s_n-1}{\ups_n}
 =J_n(s_n-1)^2+U_n.
\end{equation}
At $(\alpha,v)=(1,k_n)$, the quadratic form in~\eqref{eq:uniform-info} equals $J_n$ by~\eqref{eq:schur-quadratic}. Hence $E_n+n\norm{k_n}^2=O(J_n+\log n)$ almost surely along sequences with $r_n=O(n)$ and $R_n=o(n)$. With $G_n=O(n)$ and $\mu_n=J_n+k_n^TG_nk_n$,
\begin{equation}\label{eq:shared-mu}
 \mu_n+E_n=O(J_n+\log n).
\end{equation}

Fix $\eta>0$, put $\calB_\eta\triangleq\{n:|\bhat_{1,n}-b_1|\ge\eta\}$ and $\mathcal D_\eta\triangleq\{|\calB_\eta|=\infty\}$. Suppose, for a contradiction, that $\Prob(\mathcal D_\eta)>0$. On $\mathcal D_\eta$, let $n\to\infty$ through $\calB_\eta$.

Along $\calB_\eta$, \eqref{eq:Guo-bad} ensures $r_n=O(n)$, $R_n=o(n)$ and $V_n=O(\log n)$ almost surely. Applying~\eqref{eq:uniform-info} with $(\alpha,v)=(s_n-1,\ups_n)$ yields
\[
 n\norm{\ups_n}^2=O\!\left(V_n+(s_n-1)^2\log n\right)=O((\log n)^2).
\]
Thus $\ups_n\to0$ along $\calB_\eta$. Component $p+1$ in~\eqref{eq:coords} reads
\[
 b_1(s_n-1)=\bhat_{1,n}-b_1-e_{p+1}^TQ\ups_n,
\]
so $|s_n-1|\ge\eta/(2|b_1|)$ eventually on $\calB_\eta$. By~\eqref{eq:V-decomposition}, $V_n\ge J_n(s_n-1)^2$, so $J_n=O(\log n)$ by~\eqref{eq:Guo-bad}. By~\eqref{eq:shared-mu}, $\mu_n+E_n=O(\log n)$ along these indices.

Thus $J_n\gtrsim\log n$ by Lemma~\ref{lem:prediction-window}.
$T_n^2=O(\mu_nH_n)$ by~\cite[Theorem~1.3.10(i)]{Guo2020} applied to $\pi_t$, and $(k_n^TS_n)^2\le\mu_nU_n=O(\mu_nH_n)$ by~\eqref{eq:window-preliminaries}. Since $|s_n-1|\ge\eta/(2|b_1|)$, \eqref{eq:T} yields
\[
 J_n^2=O(\mu_nH_n)=O(\log n\,H_n).
\]
The preceding lower bound on $J_n$ would then imply
$\log n=O(H_n)$, a contradiction. Thus
$\Prob(\mathcal D_\eta)=0$. Taking $\eta=1/k$, $k\ge1$, proves
\begin{equation*}
 \bhat_{1,n}\to b_1\as
\end{equation*}
With $\bhat_{1,n}\to b_1\ne0$, \cite[Theorem~6.1 and Corollary~6.1]{Guo95} apply along the full sequence. Almost surely, $r_n=O(n)$ and $R_n=o(n)$, establishing stability and optimal tracking.

By~\eqref{eq:uniform-info}, \eqref{eq:V-decomposition} and~\cite[Corollary~3.1(ii)]{Guo95}, $n\norm{\ups_n}^2\lesssim V_n+(s_n-1)^2\log n\lesssim(1+\log n)V_n=O((\log n)^2)$. Thus ${\ups_n\to0}$. By component $p+1$ of~\eqref{eq:coords}, $\bhat_{1,n}\to b_1$ then implies $s_n\to1$, so $\thh_n\to\theta$.

For $d=1$, $p=0$ and $q=1$. Equations~\eqref{eq:rls}--\eqref{eq:info-update} give
\[
 \bhat_{1,n}-b_1=
 \frac{P_0^{-1}(\bhat_{1,0}-b_1)+\sum_{t<n}u_tw_{t+1}}
      {P_0^{-1}+\sum_{t<n}u_t^2}.
\]
If $\sum_{t\ge0}u_t^2<\infty$, inputs are bounded. Otherwise $\bhat_{1,n}\to b_1\ne0$ almost surely by~\cite[Theorem~1.3.10(i)]{Guo2020} applied to $u_t$, so $u_n=y_{n+1}^*/\bhat_{1,n}=O(1)$ by~\eqref{eq:ce}.

With $\ph_n=u_n=O(1)$, $r_n=O(n)$ and $P_nu_n^2\le P_0u_n^2=O(1)$, so $R_n=O(\log n)$ by~\cite[Corollary~3.1(i)]{Guo95}. By $y_{n+1}=b_1u_n+w_{n+1}$ and the strong law for $w_j^2$, the average output energy is bounded, proving~\eqref{eq:stability}. Since $R_n=o(n)$,~\eqref{eq:tracking} follows. Conditioning on the reference path extends the proof to random references (Appendix~\ref{app:noise}).

\section{Information Growth and Convergence Rates}\label{sec:rates}
Fix a bounded deterministic reference. Using consistency from Section~\ref{sec:identification}, we relate the information growth to the reference energy. Write~\mbox{$h_n\triangleq\ph_n^TP_n\ph_n$}.

\begin{lemma}\label{lem:bootstrap}
Under the conditions of Theorem~\ref{thm:basic}, with $d\ge2$, the following estimates hold almost surely.
\begin{equation}\label{eq:rate-start}
 \begin{gathered}
 G_n\asymp nI,\qquad J_n\asymp\mu_n\asymp E_n+\log n,\\
 R_n=O(\log n),\qquad M_n\asymp\diag(\mu_n,nI_{d-1}).
 \end{gathered}
\end{equation}
\end{lemma}
\begin{proof}
By Section~\ref{sec:identification}, $s_n\to1$, $r_n=O(n)$ and $R_n=o(n)$ almost surely. Thus $G_n=O(n)$ and $\mu_n=O(n)$. By~\cite[Corollary~3.1(ii)]{Guo95}, $V_n=O(\log n)$. With $\alpha=0$ in~\eqref{eq:uniform-info}, $G_n\gtrsim nI$, hence $G_n\asymp nI$.

If $J_n\le\log n$, then $\mu_n+E_n=O(\log n)$ by~\eqref{eq:shared-mu}. Lemma~\ref{lem:prediction-window} implies $J_n\gtrsim\log n$, which also holds when $J_n>\log n$. Since $J_n\le\mu_n$ and $\mu_n=O(J_n+\log n)=O(J_n)$, we obtain $J_n\asymp\mu_n$.

To prove the matrix comparison in~\eqref{eq:rate-start}, use $\norm{k_n}_{G_n}^2\le\mu_n\asymp J_n$, $G_n\asymp nI$, the triangle inequality and~\eqref{eq:schur-quadratic}. For $\alpha\in\R$ and $v\in\R^{d-1}$,
\[
 \begin{gathered}
 \norm{v-\alpha k_n}_{G_n}^2\lesssim\norm v_{G_n}^2+\alpha^2J_n,\qquad
 \norm v_{G_n}^2\lesssim\norm{v-\alpha k_n}_{G_n}^2+\alpha^2J_n,\\
 \binom{\alpha}{v}^{\!T}M_n\binom{\alpha}{v}
 =J_n\alpha^2+\norm{v-\alpha k_n}_{G_n}^2
 \asymp\mu_n\alpha^2\!+\!n\norm v^2,
 \end{gathered}
\]
Taking inverses gives $M_n^{-1}\asymp\diag(\mu_n^{-1},n^{-1}I)$.

To bound $h_n$, take $(\alpha,v)=(s_n-1,\ups_n)$ in the matrix bound. Then $n\norm{\ups_n}^2\lesssim V_n=O(\log n)$ by~\eqref{eq:V-decomposition}. By~\eqref{eq:ce} and~\eqref{eq:coords}, $s_n\pi_n=y_{n+1}^*-\ups_n^Tx_n$. Since $\norm{x_n}^2\le r_{n+1}=O(n)$, $s_n\to1$ and the reference is bounded,
\begin{equation}\label{eq:pi-log-bound}
 \pi_n^2\lesssim1+\norm{\ups_n}^2\norm{x_n}^2=O(\log n).
\end{equation}
In $h_n$, write $P_n=L_\theta M_n^{-1}L_\theta^T$ and $L_\theta^T\ph_n=(\pi_n,x_n^T)^T$. By the bound for $M_n^{-1}$, \eqref{eq:pi-log-bound} and $\mu_n\gtrsim\log n$,
\begin{equation}\label{eq:leverage-bounded}
 h_n\asymp\frac{\pi_n^2}{\mu_n}+\frac{\norm{x_n}^2}{n}=O(1).
\end{equation}
By~\eqref{eq:ce}, \eqref{eq:leverage-bounded} and~\cite[Corollary~3.1(i)]{Guo95},
\[
 R_n\le(1+\max_{t<n}h_t)\sum_{t<n}\frac{\varepsilon_t^2}{1+h_t}=O(\log n).
\]

Since $\pi_t=y_{t+1}^*+\varepsilon_t$,
\[
 \mu_n\le\mu_0+2E_n+2R_n,\qquad E_n\le2\mu_n+2R_n.
\]
Since $\mu_n\gtrsim\log n$, the two bounds imply $\mu_n\asymp E_n+\log n$.
\end{proof}

\begin{lemma}\label{lem:bounded-predictor}\label{lem:refined-estimates}
Under the conditions of Lemma~\ref{lem:bootstrap}, for fixed $2/\nu<\gamma<1$, almost surely,
\begin{enumerate}
\setlength{\itemsep}{0pt}
\setlength{\parsep}{0pt}
\item $\norm{x_n}^2=O(n^\gamma)$, $\varepsilon_n\to0$.
\item $T_n^2=O(\mu_n\ellfun(\mu_n))$, $S_n=O(\sqrt{nH_n})$.
\item $V_n=O(H_n)$, $\norm{\ups_n}^2=O(H_n/n)$.
\end{enumerate}
\end{lemma}
\begin{proof}
(i) The bounded reference and the absolutely summable impulse response of $\mathcal T$ give $\norm{Q^T(\ph_n^o-\ph_n^w)}=O(1)$. Apply Lemma~\ref{lem:noise-estimates} and~\eqref{eq:noise-difference}, with $R_{n+1}=O(\log n)$. Almost surely,
\[
 \norm{x_n}^2\lesssim\norm{z_n}^2+1+
 \sum_{t\le n}\norm{\ph_t-\ph_t^o}^2=O(n^\gamma).
\]
With $\norm{\ups_n}^2=O(\log n/n)$ from the proof of Lemma~\ref{lem:bootstrap}, we obtain $|\ups_n^Tx_n|^2=O(n^{\gamma-1}\log n)\to0$. By~\eqref{eq:ce}, \eqref{eq:coords} and $\varepsilon_n=\pi_n-y_{n+1}^*$, we have $s_n\varepsilon_n=(1-s_n)y_{n+1}^*-\ups_n^Tx_n$.
Since $s_n\to1$, $\varepsilon_n\to0$.

(ii) By (i), $\pi_t=y_{t+1}^*+\varepsilon_t=O(1)$. For $1<\nu_0<\min(\nu/2,2)$, $|\pi_t|^{2\nu_0}\lesssim\pi_t^2=\mu_{t+1}-\mu_t$, so comparison with $\int_{\mu_0}^\infty u^{-\nu_0}\,du$ gives
\[
 \sum_{t\ge0}\frac{|\pi_t|^{2\nu_0}}{\mu_{t+1}^{\nu_0}}<\infty.
\]
By~\eqref{eq:rate-start}, $\mu_n\to\infty$. Apply~\cite[Corollary~6]{DufloSenoussiTouati90} to $T_n=T_0+\sum_{t<n}\pi_tw_{t+1}$, using~\eqref{eq:noise} and the $\F_t$-measurability of $\pi_t$, to obtain $T_n^2=O(\mu_n\ellfun(\mu_n))$.

With $\sum_{t<n}\norm{x_t}^2=O(n)$ from~\eqref{eq:rate-start} and $\norm{x_n}^2=O(n^\gamma)$ from (i), apply~\cite[Lemma~5.2]{Guo95} to bound $S_n$.

(iii) By~\eqref{eq:V-decomposition} and the normal equation in the proof of Lemma~\ref{lem:schur},
\[
 V_n=(T_n,S_n^T)M_n^{-1}(T_n,S_n^T)^T.
\]
By~\eqref{eq:rate-start} and (ii),
\begin{equation}\label{eq:refined-error-energy}
 \mu_n(s_n-1)^2+n\norm{\ups_n}^2\lesssim V_n \lesssim T_n^2/\mu_n+\norm{S_n}^2/n=O(H_n),
\end{equation}
since $\ellfun(\mu_n)=O(H_n)$ by~\eqref{eq:rate-start} and $E_n=O(n)$.
\end{proof}

\begin{proof}[Proof of Theorem~\ref{thm:rates} and Corollary~\ref{cor:ratios}]
To sharpen the almost sure bound $(s_n-1)^2=O(H_n/\mu_n)$ from~\eqref{eq:refined-error-energy}, we estimate $g_n^T\ups_n$ in the first row of the LS normal equation in the proof of Lemma~\ref{lem:schur},
\begin{equation}\label{eq:scale-normal}
 \mu_n(s_n-1)=T_n-g_n^T\ups_n.
\end{equation}
Substitute $\pi_t=y_{t+1}^*-(s_t-1)\pi_t-\ups_t^Tx_t$ into $g_n=g_0+\sum_{t<n}\pi_tx_t$ to obtain
\begin{equation}\label{eq:g-decomposition}
 g_n=g_0+\sum_{t<n}y_{t+1}^*x_t
          -\sum_{t<n}(s_t-1)\pi_tx_t-\sum_{t<n}x_tx_t^T\ups_t.
\end{equation}
For fixed $2/\nu<\gamma<1$, apply~\eqref{eq:reference-score} with $\mu_n\asymp E_n+\log n$ from~\eqref{eq:rate-start} to bound the reference sum.
\begin{align}
 \Big\|\sum_{t<n}y_{t+1}^*x_t\Big\|
 &=O\!\left(\sqrt{E_n(n^\gamma+\mu_n)}\right)
 =O\!\left(E_n+\sqrt{E_n(n^\gamma+\log n)}\right)\notag\\
 &=O\!\left(E_n+\sqrt{\frac{n\mu_nH_n}{\log n}}\right),\label{eq:g-reference-bound}
\end{align}
Here the second bound uses $\mu_n=O(E_n+\log n)$, and the last uses $E_n=O(\mu_n)$ and $n^\gamma+\log n=o(nH_n/\log n)$.
By~\eqref{eq:refined-error-energy} and~\eqref{eq:rate-start}, $(s_t-1)^2=O(H_t/\log(e+t))$ and $\sum_{j<t}\norm{x_j}^2=O(t)$. Summation by parts yields
\[
 \sum_{t<n}(s_t-1)^2\norm{x_t}^2=O(nH_n/\log n).
\]
The Cauchy--Schwarz inequality then gives
\begin{equation}\label{eq:g-scale-bound}
 \Big\|\sum_{t<n}(s_t-1)\pi_tx_t\Big\|^2
 \le\Big(\sum_{t<n}\pi_t^2\Big)
       \Big(\sum_{t<n}(s_t-1)^2\norm{x_t}^2\Big) =O(n\mu_nH_n/\log n).
\end{equation}
Bound $\norm{\ups_t}$ by Lemma~\ref{lem:bounded-predictor}(iii). The last step uses summation by parts with~\eqref{eq:rate-start}.
\begin{equation}\label{eq:g-orthogonal-bound}
 \Big\|\sum_{t<n}x_tx_t^T\ups_t\Big\|
 \le\sum_{t<n}\norm{x_t}^2\norm{\ups_t} \lesssim\sum_{1\le t<n}\norm{x_t}^2\sqrt{H_t/t}+O(1) =O(\sqrt{nH_n}).
\end{equation}
Bound the sums in~\eqref{eq:g-decomposition} by~\eqref{eq:g-reference-bound}--\eqref{eq:g-orthogonal-bound}. With $|g_n^T\ups_n|\le\norm{g_n}\norm{\ups_n}$, $\norm{\ups_n}=O(\sqrt{H_n/n})$ and $\mu_n\gtrsim\log n$,
\begin{equation}\label{eq:scale-score-cross}
 |g_n^T\ups_n|=O\!\left(E_n\sqrt{\frac{H_n}{n}}+H_n\sqrt{\frac{\mu_n}{\log n}}\right).
\end{equation}
Since $H_n^2=o(\log n)$ and $\ellfun(\mu_n)\ge1$,
\begin{equation}\label{eq:scale-log-comparison}
 \frac{H_n^2\mu_n/\log n}{\mu_n\ellfun(\mu_n)}
 =\frac{H_n^2}{\log n\,\ellfun(\mu_n)}\longrightarrow0.
\end{equation}
Since $E_n\lesssim\mu_n\lesssim n$ and $u/\ellfun(u)$ is increasing,
\begin{equation}\label{eq:scale-reference-comparison}
 \frac{E_n^2H_n/n}{\mu_n\ellfun(\mu_n)}
 \lesssim\frac{\mu_nH_n}{n\ellfun(\mu_n)}=O(1).
\end{equation}
By~\eqref{eq:scale-score-cross}--\eqref{eq:scale-reference-comparison}, $g_n^T\ups_n=O(\sqrt{\mu_n\ellfun(\mu_n)})$. With $T_n^2=O(\mu_n\ellfun(\mu_n))$ from Lemma~\ref{lem:bounded-predictor}(ii), \eqref{eq:scale-normal} yields $(s_n-1)^2=O(\ellfun(\mu_n)/\mu_n)$.
With $\mu_n\asymp E_n+\log n$, this gives the first bound in~\eqref{eq:rate-main}. The second is Lemma~\ref{lem:bounded-predictor}(iii). Since $\mu_n=O(n)$ and $u/\ellfun(u)$ is increasing, $H_n/n=O(\ellfun(\mu_n)/\mu_n)$, so
\[
 \norm{\thh_n-\theta}^2
 =\norm\theta^2(s_n-1)^2+\norm{\ups_n}^2
 =O(\ellfun(\mu_n)/\mu_n).
\]
This proves~\eqref{eq:full-rate}. Finally, \eqref{eq:uniform-rate} follows from $\mu_n\gtrsim\log n$.

For Corollary~\ref{cor:ratios},
\[
 \begin{aligned}
 \frac{\thh_n}{\bhat_{1,n}}-\frac\theta{b_1}
 &=\frac1{\bhat_{1,n}}\left(I-\frac{\theta e_{p+1}^T}{b_1}\right)Q\ups_n,
 \end{aligned}
\]
whose squared norm is $O(H_n/n)$ by Lemma~\ref{lem:bounded-predictor} and $\bhat_{1,n}\to b_1\ne0$.
\end{proof}

\section{Logarithm Law and Sharpness}\label{sec:cost}\label{sec:residual-proof}
\begin{proof}[Proof of Theorem~\ref{thm:cost}]
Fix $2/\nu<\gamma<1$ and $1<\nu_0<\min(\nu/2,2)$. By Lemma~\ref{lem:bounded-predictor}(i) and $\pi_n=y_{n+1}^*+\varepsilon_n$, $\norm{x_n}^2=O(n^\gamma)$ and $\pi_n=O(1)$. Thus $h_n=O(1/\log n)$ by~\eqref{eq:rate-start} and~\eqref{eq:leverage-bounded}. The bound on $\norm{x_n}^2$ gives the first inequality below. Summation by parts using $\sum_{t<n}\norm{x_t}^2=O(n)$ from~\eqref{eq:rate-start} gives convergence of the series on its right. The second series converges by the integral comparison in the proof of Lemma~\ref{lem:bounded-predictor}(ii) and $\mu_{n+1}\asymp\mu_n$.
\begin{equation}\label{eq:leverage-summability}
 \sum_{n\ge1}\left(\frac{\norm{x_n}^2}{n}\right)^{\nu_0}
 \lesssim\sum_{n\ge1}\frac{\norm{x_n}^2}{n^{1+(1-\gamma)(\nu_0-1)}}<\infty,\qquad
 \sum_{n\ge0}\left(\frac{\pi_n^2}{\mu_n}\right)^{\nu_0}<\infty.
\end{equation}
By~\eqref{eq:leverage-bounded} and~\eqref{eq:leverage-summability}, $\sum_{n\ge0}h_n^{\nu_0}<\infty$. Since $V_n=O(H_n)$ by Lemma~\ref{lem:bounded-predictor}, $h_n^{2-\nu_0}V_n\to0$ and thus $\sum_{n\ge0}h_n^2(1+V_n)<\infty$.

By~\eqref{eq:ce}, $\varepsilon_n=-(\thh_n-\theta)^T\ph_n$, so the Cauchy--Schwarz inequality gives $\varepsilon_n^2\le h_nV_n$. By~\eqref{eq:rls}--\eqref{eq:info-update} and $\Delta\ell_n=\log(1+h_n)$,
\begin{align}
 \Delta V_n
 &=\frac{h_nw_{n+1}^2-\varepsilon_n^2-2\varepsilon_nw_{n+1}}{1+h_n}\notag\\
 &=-\varepsilon_n^2-\frac{2\varepsilon_nw_{n+1}}{1+h_n}
 +\sigma^2\Delta\ell_n
 +\frac{h_n(w_{n+1}^2-\sigma^2)}{1+h_n}
 +O(h_n^2(1+V_n)).\label{eq:ls-energy-expansion}
\end{align}
The remainder follows from $\varepsilon_n^2 h_n/(1+h_n)\le h_n^2V_n$ and $h_n/(1+h_n)-\log(1+h_n)=O(h_n^2)$.
By~\eqref{eq:noise}, the martingale difference $h_n(w_{n+1}^2-\sigma^2)/(1+h_n)$ has conditional $\nu_0$th absolute moment $O(h_n^{\nu_0})$. Since $\sum_{n\ge0}h_n^{\nu_0}<\infty$, the martingale series converges by~\cite[Theorem~2.17]{HallHeyde80}. Summing~\eqref{eq:ls-energy-expansion} and using the summable remainder gives almost sure convergence of
\begin{equation}\label{eq:cost-convergent}
 R_n+2\sum_{t<n}\frac{\varepsilon_tw_{t+1}}{1+h_t}+V_n-\sigma^2\ell_n.
\end{equation}
Using $\det M_n=J_n\det G_n$ and~\eqref{eq:rate-start}, we obtain
\begin{equation}\label{eq:logdet-expansion}
 \ell_n=\log(J_n/J_0)+\log(\det G_n/\det G_0) =(d-1)\log n+\log(E_n+\log n)+O(1).
\end{equation}
If $R_n$ were bounded, $\sum_{t\ge0}(\varepsilon_t/(1+h_t))^2<\infty$. Then the noise sum in~\eqref{eq:cost-convergent} would converge by~\cite[Theorem~2.17]{HallHeyde80}. Since $V_n=O(H_n)$, convergence of~\eqref{eq:cost-convergent} would imply $\ell_n=O(H_n)$, contradicting~\eqref{eq:logdet-expansion}. Hence $R_n\to\infty$ by monotonicity.

The $\F_t$-measurable coefficients $\varepsilon_t/(1+h_t)$ are bounded by Lemma~\ref{lem:bounded-predictor}(i). Since $h_t$ is bounded,
\[
 \sum_{t<n}\frac{\varepsilon_t^2}{(1+h_t)^2}\asymp R_n\to\infty,
\]
with $R_n=O(\log n)$ by~\eqref{eq:rate-start}. The integral comparison in the proof of Lemma~\ref{lem:bounded-predictor}(ii) verifies the summability condition of~\cite[Corollary~6]{DufloSenoussiTouati90}, which bounds the noise sum in~\eqref{eq:cost-convergent} by $O(\sqrt{\log n\,\ellfun(\log n)})$. Substituting this bound into~\eqref{eq:cost-convergent} and using~\eqref{eq:logdet-expansion} proves~\eqref{eq:R-ell}.
\end{proof}

\begin{proof}[Proof of Corollary~\ref{cor:sharp}]
Under zero reference, $\pi_n=\varepsilon_n$. Thus by~\eqref{eq:R-ell}, almost surely,
\begin{equation}\label{eq:zero-reference-mu}
 \mu_n=\mu_0+R_n\sim(d-1)\sigma^2\log n.
\end{equation}
With $E_n=0$, $\mu_n(s_n-1)=T_n+O(H_n)$ by~\eqref{eq:scale-normal}, \eqref{eq:scale-score-cross} and~\eqref{eq:zero-reference-mu}. For $1<\nu_0<\min(\nu/2,2)$, the second bound in~\eqref{eq:leverage-summability} and~\eqref{eq:zero-reference-mu} imply
\[
 \sum_{t\ge1}\left(\frac{|\pi_t|}{\sqrt{\log(e+t)}}\right)^{2\nu_0}<\infty.
\]
By~\eqref{eq:noise} and~\eqref{eq:zero-reference-mu}, apply~\cite[Lemma~C.2]{Bercu98} with $c_n=\log(e+n)$ to obtain
\begin{equation}\label{eq:scale-score-lil}
 \limsup_{n\to\infty}\frac{T_n^2}{2\sigma^2\mu_n\log\log\mu_n}=1.
\end{equation}
Since $H_n=o(\sqrt{\mu_n\log\log\mu_n})$,~\eqref{eq:zero-reference-mu}--\eqref{eq:scale-score-lil} yield~\eqref{eq:scale-sharp}, using $\log\log\mu_n\sim\log\log\log n$. By Lemma~\ref{lem:bounded-predictor}(iii), $\norm{\ups_n}^2=O(H_n/n)=o(\log\log\log n/\log n)$, so the orthogonal decomposition~\eqref{eq:coords} yields~\eqref{eq:vector-sharp}. The limit superior for each coordinate follows from $(\thh_n-\theta)_j=\theta_j(s_n-1)+(Q\ups_n)_j$.
\end{proof}

\section{Proof of the Distributional Limits}\label{sec:distribution-proof}

Fix a bounded deterministic reference and assume the conditions of Theorem~\ref{thm:distribution}. Let $C_\ell$ be the filter coefficients in $z_t=\sum_{\ell\ge0}C_\ell w_{t-\ell}$, with $w_t=0$ for $t\le0$. Put $\Omega\triangleq\sigma^2\sum_{\ell\ge0}C_\ell C_\ell^T$ and $d_t\triangleq Q^T(\ph_t^o-\ph_t^w)$. The comparison matrix for $n\ge3$ is
\begin{equation}\label{eq:distribution-information}
 \bar M_n\triangleq
 \diag\bigl((d-1)\sigma^2\log n,n\Omega\bigr)
 +\sum_{t<n}\binom{y_{t+1}^*}{d_t}\binom{y_{t+1}^*}{d_t}^{\!T}.
\end{equation}
We first establish the limits of the blocks of $M_n$ relative to the corresponding blocks of $\bar M_n$.

\begin{lemma}\label{lem:distribution-info}
Almost surely,
\begin{equation}\label{eq:distribution-matrix-limit}
 \bar M_n\asymp\diag(E_n+\log n,nI_{d-1}),\qquad
 \bar M_n^{-1/2}M_n\bar M_n^{-1/2}\longrightarrow I_d.
\end{equation}
\end{lemma}
\begin{proof}
For the upper-left entry,
\begin{equation}\label{eq:distribution-mu}
 \mu_n=\mu_0+E_n+R_n+2\sum_{t<n}y_{t+1}^*\varepsilon_t.
\end{equation}
By~\eqref{eq:R-ell}, $R_n=(d-1)\sigma^2\log n+o(E_n+\log n)$. For indices with $E_n>K\log n$, the Cauchy--Schwarz inequality bounds the cross sum in~\eqref{eq:distribution-mu}, divided by $E_n+\log n$, by $K^{-1/2}\sqrt{R_n/\log n}$. We show that the cross sum is $o(\log n)$ along the indices with $E_n\le K\log n$ for each fixed $K\ge1$.

On those indices, $\mu_n=O(\log n)$ by~\eqref{eq:rate-start}. For fixed $2/\nu<\gamma<1$, choose $A>\max\{2,(1-\gamma)^{-1}\}$ and put $m=\lceil(\log n)^A\rceil$. The sum over $t<m$ is $O(\sqrt{E_nR_m})=o(\log n)$ by~\eqref{eq:rate-start}. For the sum over $m\le t<n$, use~\eqref{eq:reference-prediction-decomposition} to write $\varepsilon_t=-\widehat f_t-(s_t-1)\xi_t$. Since $\lambda_t\to0$ by Lemma~\ref{lem:bounded-predictor}(i) and $G_t\asymp tI$, \eqref{eq:schur-increment} gives $\sum_{m\le t<n}\xi_t^2=O(J_n)=O(\log n)$. Hence
\[
 \left|\sum_{m\le t<n}y_{t+1}^*(s_t-1)\xi_t\right|
 \le \sup_{t\ge m}|s_t-1|\sqrt{E_n\sum_{m\le t<n}\xi_t^2}
 =o(\log n).
\]

To bound the sum involving $\widehat f_t$, use $\zeta_t$, $W_t$ and $c_t$ from the proof of Lemma~\ref{lem:prediction-window}. Write $a_t=\zeta_t^TG_t^{-1}x_t$, so that $\bar f_t=\widehat f_t+a_t$ and $y_{t+1}^*-c_t=(y_{t+1}^*-a_t)/(1+\lambda_t)$ by the update of $G_t^{-1}$. Subtract the second identity in~\eqref{eq:reference-energy-updates} from the first to get
\begin{equation}\label{eq:distribution-polarization}
 \Delta(\zeta_t^T\rho_t)
 = c_tw_{t+1}+\frac{\widehat f_t(y_{t+1}^*-a_t)}{1+\lambda_t}.
\end{equation}
By~\eqref{eq:reference-score} and $G_t\asymp tI$,
\[
 \max_{m\le t\le n}W_t
 =O\bigl((\log n)m^{\gamma-1}+(\log n)^2/m\bigr)=o(1/H_n).
\]
Since $U_t\le V_t=O(H_t)$ by~\eqref{eq:V-decomposition} and Lemma~\ref{lem:bounded-predictor}(iii), both endpoint values of $\zeta_t^T\rho_t$ are $o(1)$. Equation~\eqref{eq:reference-coefficient-energy} and~\cite[Theorem~1.3.10(i)]{Guo2020} bound the noise sum in~\eqref{eq:distribution-polarization} by $O(\sqrt{(1+E_n)\log(e+E_n)})=o(\log n)$ over this interval. Also $a_t^2\le W_t\lambda_t$, and summing $\lambda_t/(1+\lambda_t)\le\log(\det G_{t+1}/\det G_t)$ gives
\[
 \sum_{m\le t<n}\frac{a_t^2}{1+\lambda_t}=o(\log n).
\]
The prediction energy in~\eqref{eq:prediction-energy-identity} is $O(\log n)$ since $G_n\asymp nI$ and $U_n\ge0$. Summing~\eqref{eq:distribution-polarization} and applying the Cauchy--Schwarz inequality to the term involving $a_t$ gives $\sum_{m\le t<n}\widehat f_t y_{t+1}^*/(1+\lambda_t)=o(\log n)$. The effect of removing the denominator is at most
\[
 \max_{m\le t<n}\lambda_t
 \left(\sum_{m\le t<n}\frac{\widehat f_t^2}{1+\lambda_t}\right)^{1/2}
 \sqrt{E_n}
 =o(\log n).
\]
Together with the bounds for $t<m$ and $(s_t-1)\xi_t$, this proves the cross sum estimate for $E_n\le K\log n$. Taking $K\to\infty$ through the integers in the complementary bound yields $\sum_{t<n}y_{t+1}^*\varepsilon_t=o(E_n+\log n)$ on the full sequence. Thus
\begin{equation}\label{eq:distribution-mu-equiv}
 \mu_n\sim E_n+(d-1)\sigma^2\log n.
\end{equation}

For the lower-right block, $n^{-1}\sum_{t<n}z_tz_t^T\to\Omega$ almost surely. To see this, first truncate the filter to finitely many lags. The empirical mean of each product of two distinct lags of $w_t$ tends to zero by the martingale strong law~\cite[Theorems~2.17 and~2.18]{HallHeyde80}, indexing the products by the more recent noise. Indeed, for $1<\nu_0<\min(\nu/2,2)$ the conditional $\nu_0$th moments are bounded almost surely by a constant times the $\nu_0$th powers of the older noise. The sum of these powers up to $n$ is $O(n)$ by the strong law for $w_t^2$, so division by $t^{\nu_0}$ yields a summable sequence by summation by parts. Equal lags give the limit $\sigma^2$ by that same law for squares. The root mean square of the omitted filter terms is bounded by $\sum_{\ell>L}\norm{C_\ell}(n^{-1}\sum_{t<n}w_t^2)^{1/2}$ for a truncation at lag $L$. As $L\to\infty$, this tends to zero uniformly in $n$ almost surely, so the finite filter limits give the asserted limit. Lemma~\ref{lem:noise-estimates} then gives $\Omega>0$.

Write $x_t=z_t+d_t+\delta_t$ with $\delta_t=Q^T(\ph_t-\ph_t^o)$. By the filter bound and~\eqref{eq:noise-difference}, $\sum_{t<n}\norm{d_t}^2=O(E_n)$ and $\sum_{t<n}\norm{\delta_t}^2=O(\log n)$. The terms with $\delta_t$ in $G_n=G_0+\sum_{t<n}x_tx_t^T$ are $O(\sqrt{n\log n}+\log n)$ by the Cauchy--Schwarz inequality. With $Z_n$, $D_n$ and $\Pi_n$ as in the proof of Lemma~\ref{lem:uniform-regression}, the mixed term $\sum_{t<n}z_td_t^T$ equals $(\Pi_n Z_n)^TD_nQ$. Its norm is $O(\sqrt{n^\gamma E_n})=o(n)$ by~\eqref{eq:projected-noise}. Thus
\[
 G_n-n\Omega-\sum_{t<n}d_td_t^T=o(n).
\]

For the off-diagonal block, use $g_n=g_0+\sum_{t<n}(y_{t+1}^*+\varepsilon_t)x_t$. Equation~\eqref{eq:weighted-filter} and the bound on $\sum_{t<n}\norm{\delta_t}^2$ give
\[
 \left\|\sum_{t<n}y_{t+1}^*(x_t-d_t)\right\|
 =O(\sqrt{E_nn^\gamma}+\sqrt{E_n\log n})
 =o\!\left(\sqrt{n(E_n+\log n)}\right).
\]
For large $t$, $|s_t|\ge1/2$ and $s_t\varepsilon_t=(1-s_t)y_{t+1}^*-\ups_t^Tx_t$ by~\eqref{eq:ce} and~\eqref{eq:coords}. Since $s_t\to1$, splitting off any fixed initial segment shows that $\sum_{t<n}(1-s_t)^2(y_{t+1}^*)^2=o(E_n+\log n)$. With $\sum_{t<n}\norm{x_t}^2=O(n)$, the Cauchy--Schwarz inequality bounds the contribution of $(1-s_t)y_{t+1}^*/s_t$ to $\sum_{t<n}\varepsilon_tx_t$ by $o(\sqrt{n(E_n+\log n)})$. For the other term, Lemma~\ref{lem:bounded-predictor}(iii) and summation by parts give
\[
 \sum_{1\le t<n}\norm{\ups_t}\norm{x_t}^2
 =O\!\left(\sqrt{H_n}\sum_{1\le t<n}\frac{\norm{x_t}^2}{\sqrt t}\right)
 =O(\sqrt{nH_n})=o\!\left(\sqrt{n(E_n+\log n)}\right).
\]
Hence $g_n-\sum_{t<n}y_{t+1}^*d_t=o(\sqrt{n(E_n+\log n)})$.

Put $B_n=\diag(E_n+\log n,nI_{d-1})$. The three block limits show $B_n^{-1/2}(M_n-\bar M_n)B_n^{-1/2}\to0$. Since $M_n\asymp B_n$ by~\eqref{eq:rate-start}, $\bar M_n\asymp B_n$ and $\norm{\bar M_n^{-1/2}B_n^{1/2}}=O(1)$. This gives~\eqref{eq:distribution-matrix-limit}.
\end{proof}

\begin{proof}[Proof of Theorem~\ref{thm:distribution}]
Put $v_{n,t}=\bar M_n^{-1/2}(\pi_t,x_t^T)^T$ for $0\le t<n$. For the deterministic reference, $v_{n,t}$ is $\F_t$-measurable. The conditional covariance of $\sigma^{-1}v_{n,t}w_{t+1}$ is $v_{n,t}v_{n,t}^T$ by~\eqref{eq:noise}. Lemma~\ref{lem:distribution-info} gives
\[
 \sum_{t<n}v_{n,t}v_{n,t}^T
 =\bar M_n^{-1/2}(M_n-M_0)\bar M_n^{-1/2}\longrightarrow I_d\as
\]
By Lemma~\ref{lem:bounded-predictor}(i), $\pi_t=y_{t+1}^*+\varepsilon_t$ is bounded and $\norm{x_t}^2=O(t^\gamma)$ for a fixed $\gamma<1$. The bound on $\bar M_n^{-1}$ in Lemma~\ref{lem:distribution-info} thus gives
\[
 \max_{t<n}\norm{v_{n,t}}^2
 =O\!\left(\frac1{E_n+\log n}+n^{\gamma-1}\right)=o(1)\as
\]
For each $\eta>0$, the conditional Lindeberg sum is at most
\[
 \frac{\eta^{2-\nu}}{\sigma^\nu}
 \sup_{t\ge0}\E[|w_{t+1}|^\nu\mid\F_t]\,
 \left(\max_{t<n}\norm{v_{n,t}}\right)^{\nu-2}
 \sum_{t<n}\norm{v_{n,t}}^2=o(1)\as
\]
To use square integrable rows, keep only terms with $\sum_{j=0}^{t}\norm{v_{n,j}}^2\le2d$. This event is $\F_t$-measurable, and the retained coefficients have total squared norm at most $2d$. The probability that any term is removed tends to zero because the trace of the covariance sum tends to $d$. The truncated rows have the same covariance and Lindeberg limits. Applying the martingale central limit theorem~\cite[Corollary~3.1, pp.~58--59]{HallHeyde80} to each scalar projection, with filtration $\{\F_t\}$ in every row, gives the vector normal limit by the Cram\'er--Wold theorem. The initial term in $(T_n,S_n^T)^T$ vanishes after normalization since $\bar M_n^{-1/2}\to0$. The LS normal equation in the proof of Lemma~\ref{lem:schur} therefore gives
\[
 \sigma^{-1}\bar M_n^{1/2}\binom{s_n-1}{\ups_n}
 =\bigl(\bar M_n^{-1/2}M_n\bar M_n^{-1/2}\bigr)^{-1}
   \sigma^{-1}\bar M_n^{-1/2}\binom{T_n}{S_n}
 \xrightarrow{d}\mathcal N_d(0,I_d).
\]
By~\eqref{eq:V-decomposition} and~\eqref{eq:distribution-matrix-limit}, $V_n/\sigma^2$ differs from the squared norm of this normalized error by $o_{\Prob}(1)$. Since $\widehat\sigma_n^2\to\sigma^2>0$ almost surely by~\eqref{eq:tracking} and the noise strong law, Slutsky's theorem gives~\eqref{eq:chisquare}.

For a random reference, condition on its complete path as in Appendix~\ref{app:noise}. For almost every reference path both limiting distributions are the same as in the deterministic case. Applying bounded convergence to conditional expectations of bounded continuous functions gives~\eqref{eq:normal-limit} and~\eqref{eq:chisquare} under the joint law.
\end{proof}

\section{Numerical Illustration}\label{sec:numerical}

We illustrate how reference energy affects parameter estimation and tracking. Consider the open-loop unstable system $y_{n+1}=1.1y_n+u_n+w_{n+1}$ with independent standard Gaussian noise, $y_0=0$, $\thh_0=(0.2,1.5)^T$ and $P_0=I_2/4$. We run~\eqref{eq:rls}--\eqref{eq:ce} for $10^6$ steps over 200 independent noise paths, using the same paths for both references. For the zero and unit references, $E_n=0$ and $E_n=n$, respectively. Under unit reference, the lagged vectors eventually coincide, so Guo's reference condition fails (see the remark after Theorem~\ref{thm:rates}).

\begin{figure}[H]
\centering
\begin{minipage}[c]{.68\linewidth}
\includegraphics[width=\linewidth]{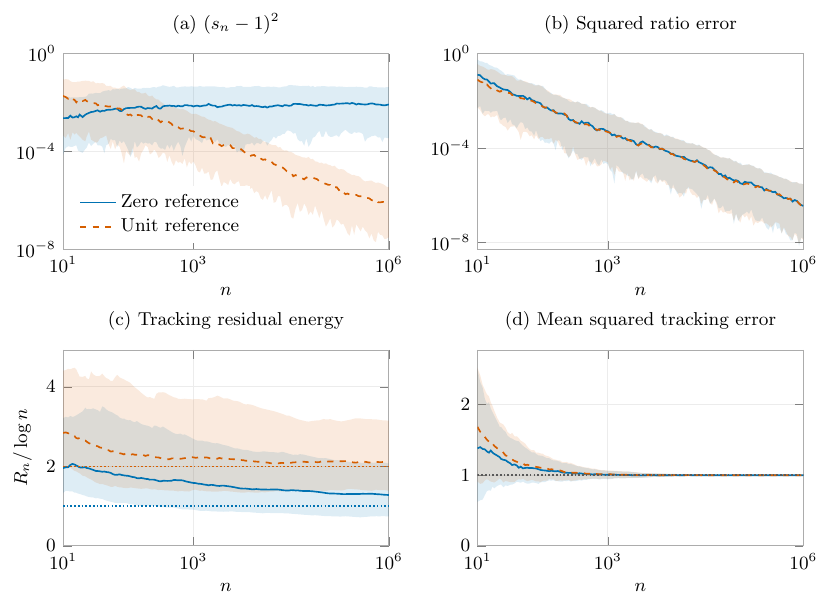}
\end{minipage}\hfill
\begin{minipage}[c]{.30\linewidth}
\setlength{\captionindent}{0pt}
\caption[Parameter estimation and tracking with zero and unit references.]{\raggedright Parameter estimation and tracking with zero and unit references. Panel~(a) shows \mbox{$(s_n-1)^2$}, and panel~(b) shows \mbox{$|(\thh_n)_1/\bhat_{1,n}-1.1|^2$}. In~(d), squared tracking errors are averaged over time on each path. Curves are medians and bands span the 10th--90th percentiles across paths. Dotted lines mark the limits $1$ and $2$ in~(c) and the noise variance $1$ in~(d).}
\label{fig:learning}
\end{minipage}
\end{figure}

In Fig.~\ref{fig:learning}(a), the median of $(s_n-1)^2$ falls under unit reference and varies little under zero reference at large $n$. At $n=10^6$, these medians are $6.8\times10^{-7}$ and $8.3\times10^{-3}$, respectively. The squared ratio error curves nearly coincide (Fig.~\ref{fig:learning}(b)). Theorem~\ref{thm:rates} bounds $(s_n-1)^2$ more tightly under the unit reference, while the squared ratio error bound in Corollary~\ref{cor:ratios} is the same for both references.

Panel~(c) compares $R_n/\log n$ with its limits $1$ and $2$ from Corollary~\ref{cor:sharp} and Theorem~\ref{thm:cost}. In both cases $R_n/n\to0$, so the time average in panel~(d) tends to $\sigma^2=1$ by~\eqref{eq:tracking}. Under zero reference, the control law uses only the ratio $(\thh_n)_1/\bhat_{1,n}$. Its small error explains why tracking remains accurate despite the larger error in $s_n$.

\FloatBarrier
\section{Conclusion}\label{sec:conclusion}
We have established almost sure stability and optimal tracking of the LS-based STR with unknown input gain for every bounded reference under the stated assumptions, provided that the recursion is well defined. With at least two estimated parameters, all coefficient estimates are strongly consistent. The estimates of coefficient ratios relative to the input gain can converge faster than the full parameter estimate while the system achieves optimal tracking. Reference energy affects both the convergence rate bound for the full parameter estimate and the leading term in the logarithm law for cumulative squared tracking residuals. The parameter error bound is sharp under zero reference. For plants with at least two estimated parameters, the limiting distribution of the least squares estimate provides asymptotically exact confidence ellipsoids.
\appendix
\section{Probability and LS Estimates}\label{app:algebra}\label{app:auxiliary}
Under Assumptions~\ref{ass:noise}--\ref{ass:plant} and~\eqref{eq:well-defined}, we first establish estimates for bounded deterministic references.

For the moment bounds below, fix $K>0$. Let $w_n^{(K)}$ agree with $w_n$ until the first $t$ with $\E[|w_{t+1}|^\nu\mid\F_t]>K$, and set $w_{t+1}^{(K)}$ and all later terms to zero. This is a martingale difference sequence for $\{\F_n\}$ with conditional $\nu$th absolute moments at most $K$. Define $z_n^{(K)}$ in the same way as $z_n$ in Section~\ref{sec:filter-representations}, with $w_n^{(K)}$ in place of $w_n$. By~\eqref{eq:noise}, the probability of no stopping tends to one as $K\to\infty$.

For $n\ge1$ and deterministic $c_1,\ldots,c_n\in\R$, Burkholder's inequality~\cite[Theorem~2.10]{HallHeyde80} and Minkowski's inequality give, with $C_\nu$ depending only on $\nu$,
\begin{equation}\label{eq:noise-moment}
 \E\left|\sum_{j=1}^nc_jw_j^{(K)}\right|^\nu\le C_\nu K\left(\sum_{j=1}^nc_j^2\right)^{\nu/2}.
\end{equation}

\begin{lemma}\label{lem:noise-estimates}
For $d\ge2$ and fixed $2/\nu<\gamma<1$, almost surely,
\begin{equation*}
 \sum_{t<n}z_tz_t^T\asymp nI_{d-1},\;
 S_n^w=O(\sqrt{nH_n}),\;\norm{z_n}^2=O(n^\gamma).
\end{equation*}
\end{lemma}
\begin{proof}
With input $u_n^w$ and output $w_n$, the plant has zero tracking residual at zero reference. Apply~\cite[Lemma~5.1]{Guo95} to $\ph_n^w$ with $u_n^w$ deleted. Since $b_1\ne0$, solve $\theta^T\ph_t^w=0$ for the deleted coordinate. Thus $z_t$ is an invertible linear image of the reduced regressor, so $\sum_{t<n}z_tz_t^T\gtrsim nI_{d-1}$. By~\eqref{eq:noise} and~\cite[Theorem~2.18]{HallHeyde80}, $\sum_{t<n}w_t^2=O(n)$ almost surely. Absolute summability of the impulse response from $w_n$ to $z_n$ implies $\sum_{t<n}z_tz_t^T=O(n)$.

By~\eqref{eq:noise-moment} and absolute summability of the impulse response, $\E\norm{z_n^{(K)}}^\nu=O(1)$. Since $\gamma\nu/2>1$, $\norm{z_n^{(K)}}^2=O(n^\gamma)$ almost surely by Markov's inequality and the Borel--Cantelli lemma. The same bound holds for $z_n$ almost surely because $\Prob\{z_t^{(K)}=z_t\text{ for all }t\ge0\}\to1$ as $K\to\infty$. The bounds for $\sum_{t<n}z_tz_t^T$ and $\norm{z_n}^2$ allow us to apply~\cite[Lemma~5.2]{Guo95} to each coordinate of $z_n$ and obtain $S_n^w=O(\sqrt{nH_n})$.
\end{proof}

\begin{lemma}\label{lem:uniform-regression}
For $d\ge2$ and fixed $2/\nu<\gamma<1$, almost surely,
\begin{equation}\label{eq:weighted-filter}
 \Big\|\sum_{t<n}y_{t+1}^*z_t\Big\|^2=O(E_nn^\gamma),
\end{equation}
and, along any sequence of indices $n\to\infty$ with $r_n=O(n)$ and $R_n=o(n)$, uniformly for $\alpha\in\R$ and $v\in\R^{d-1}$,
\begin{equation}\label{eq:uniform-info}
 E_n\alpha^2+n\norm v^2\lesssim
 \binom{\alpha}{v}^{\!T}M_n\binom{\alpha}{v}+\alpha^2\log n.
\end{equation}
\end{lemma}
\begin{proof}
Let $Z_n,D_n$ have rows $z_t^T,(\ph_t^o-\ph_t^w)^T$, respectively, for $0\le t<n$. Let $\Pi_n$ be the orthogonal projection onto the column space of $D_n$. By the definition of $\ph_t^o$ in Section~\ref{sec:filter-representations},
\[
 D_n\theta=(y_{t+1}^*)_{t<n},\qquad E_n=\norm{D_n\theta}^2.
\]
Since $\Pi_nD_n=D_n$, the sum in~\eqref{eq:weighted-filter} is $(\Pi_nZ_n)^TD_n\theta$.

Let $Z_n^{(K)}$ have rows $(z_t^{(K)})^T$ and let $\{C_\ell\}_{\ell\ge0}$ be the impulse response from $w_t$ to $z_t$, so $z_t^{(K)}=\sum_{\ell\ge0}C_\ell w_{t-\ell}^{(K)}$. For a deterministic unit vector $c=(c_0,\ldots,c_{n-1})^T\in\R^n$, put $\beta_{n,j}\triangleq\sum_{j\le t<n}c_tC_{t-j}$ for $1\le j<n$. Reorder the sums and apply Young's convolution inequality to the coefficients:
\[
 \begin{aligned}
 (c^TZ_n^{(K)})^T&=\sum_{t<n}c_t\!\sum_{1\le j\le t}C_{t-j}w_j^{(K)}\!=\!\sum_{1\le j<n}\beta_{n,j}w_j^{(K)},\\
 \sum_{1\le j<n}\norm{\beta_{n,j}}^2&\le\left(\sum_{\ell\ge0}\norm{C_\ell}\right)^2\norm c^2.
 \end{aligned}
\]
Since $\sum_{\ell\ge0}\norm{C_\ell}<\infty$, applying~\eqref{eq:noise-moment} componentwise yields $\E\norm{c^TZ_n^{(K)}}^\nu=O(1)$ uniformly in $c,n$. Apply this bound to an orthonormal basis of the range of the deterministic projection $\Pi_n$. The basis has at most $d$ vectors, so $\E\norm{\Pi_nZ_n^{(K)}}_F^\nu=O(1)$. As $\nu\gamma/2>1$, $\norm{\Pi_nZ_n^{(K)}}_F^2=O(n^\gamma)$ almost surely by Markov's inequality and the Borel--Cantelli lemma. Since $\Prob\{Z_n^{(K)}=Z_n\text{ for all }n\ge0\}\to1$ as $K\to\infty$,
\begin{equation}\label{eq:projected-noise}
 \norm{\Pi_nZ_n}_F^2=O(n^\gamma)=o(n).
\end{equation}
The Cauchy--Schwarz inequality gives $\norm{Z_n^TD_n\theta}^2\le E_n\norm{\Pi_nZ_n}_F^2$, which with~\eqref{eq:projected-noise} proves~\eqref{eq:weighted-filter}.

To bound $n\norm v^2$ in~\eqref{eq:uniform-info}, we project out the reference contribution. By Lemma~\ref{lem:noise-estimates}, $\norm{Z_nv}^2\gtrsim n\norm v^2$, while $\norm{\Pi_nZ_nv}^2=o(n)\norm v^2$ by~\eqref{eq:projected-noise}. Put $x_t^o\triangleq Q^T\ph_t^o$. The column vector $(\alpha y_{t+1}^*+v^Tx_t^o)_{t<n}$, ordered by $t=0,\ldots,n-1$, equals $Z_nv+D_n(\alpha\theta+Qv)$. Since $(I-\Pi_n)D_n=0$,
\begin{equation}\label{eq:projected-info-lower}
 \begin{aligned}[b]
 n\norm v^2&\lesssim\norm{Z_nv}^2-\norm{\Pi_nZ_nv}^2
 =\norm{(I-\Pi_n)(\alpha y_{t+1}^*+v^Tx_t^o)_{t<n}}^2\\
 &\le\sum_{t<n}(\alpha y_{t+1}^*+v^Tx_t^o)^2.
 \end{aligned}
\end{equation}
The bounded reference and absolute summability of the impulse response of $\mathcal T$ imply $\norm{x_t^o-z_t}=O(1)$. By the bounds on $z_t$ in Lemma~\ref{lem:noise-estimates}, $\sum_{t<n}\norm{x_t^o}^2=O(n)$ and $\max_{t<n}\norm{x_t^o}^2=O(n^\gamma)$.

Writing $\alpha y_{t+1}^*=(\alpha y_{t+1}^*+v^Tx_t^o)-v^Tx_t^o$ and using~\eqref{eq:projected-info-lower}, we obtain
\begin{equation}\label{eq:unweighted-reference-info}
 E_n\alpha^2+n\norm v^2\le2\sum_{t<n}(\alpha y_{t+1}^*+v^Tx_t^o)^2+O(n)\norm v^2 \lesssim\sum_{t<n}(\alpha y_{t+1}^*+v^Tx_t^o)^2.
\end{equation}
Write $h_t=\ph_t^TP_t\ph_t$. By~\eqref{eq:info-update}, $\det P_{t+1}^{-1}=\det P_t^{-1}(1+h_t)$ and $\norm{P_n^{-1}}=O(r_n)$. Since $h_t/(1+h_t)\le\log(1+h_t)$ and $r_n=O(n)$,
\[
 \sum_{t<n}\frac{h_t}{1+h_t}
 \le\log\frac{\det P_n^{-1}}{\det P_0^{-1}}
 \le d\log r_n+O(1)=O(\log n).
\]
Since $y_{t+1}^*=\pi_t-\varepsilon_t$, the Cauchy--Schwarz inequality gives
\begin{align*}
 \sum_{t<n}\frac{(\alpha y_{t+1}^*+v^Tx_t^o)^2}{1+h_t}
 &\lesssim\sum_{t<n}\frac{(\alpha\pi_t+v^Tx_t)^2+\alpha^2\varepsilon_t^2}{1+h_t}
   +\norm v^2\sum_{t<n}\norm{x_t^o-x_t}^2\\
 &\lesssim\binom{\alpha}{v}^{\!T}M_n\binom{\alpha}{v}
   +\alpha^2\log n+(1+R_n)\norm v^2.
\end{align*}
The second inequality uses~\eqref{eq:info-update} for the regressor term, \eqref{eq:ce} and~\cite[Corollary~3.1(i)]{Guo95} for the residual term, and~\eqref{eq:noise-difference} for $\sum_{t<n}\norm{x_t^o-x_t}^2$. Since the reference is bounded,
\begin{equation}\label{eq:weighted-ideal-comparison}
 \begin{aligned}[b]
 \sum_{t<n}\frac{h_t(\alpha y_{t+1}^*+v^Tx_t^o)^2}{1+h_t}
 &\lesssim\left(\alpha^2+\norm v^2\max_{t<n}\norm{x_t^o}^2\right)
   \sum_{t<n}\frac{h_t}{1+h_t}\\
 &\lesssim(\alpha^2+n^\gamma\norm v^2)\log n.
 \end{aligned}
\end{equation}
Since $1/(1+h_t)+h_t/(1+h_t)=1$, adding~\eqref{eq:weighted-ideal-comparison} to the preceding bound controls the sum in~\eqref{eq:unweighted-reference-info}. Absorbing $O(1+R_n+n^\gamma\log n)\norm v^2=o(n)\norm v^2$ into $n\norm v^2$ gives~\eqref{eq:uniform-info}.
\end{proof}

\begin{lemma}\label{lem:bad-gain-estimates}
Fix $d\ge2$, $\eta>0$ and $2/\nu<\delta<1$. Almost surely on $\mathcal D_\eta$ in Section~\ref{sec:gain-information}, as $n\to\infty$ through $\calB_\eta$,
\begin{equation}\label{eq:Guo-bad}
 r_n=O(n),\quad R_n=O(n^\delta),\quad V_n=O(\log n).
\end{equation}
\end{lemma}
\begin{proof}
Choose $2/\nu<\gamma<\delta$. By~\eqref{eq:noise} and the conditional Borel--Cantelli lemma, $\max_{t\le n}w_t^2=O(n^{\gamma/2+1/\nu})$ a.s. The initial history term in the proof of Lemma~\ref{lem:geometry} decays exponentially and adds $O(1)$ to the regressor bound in~\cite[eq.~(A.10)]{Guo95}. If $r_n\to\infty$, then along $\calB_\eta$, $\sqrt{\log r_n}|\bhat_{1,n}-b_1|\ge\eta\sqrt{\log r_n}\to\infty$. For large $n\in\calB_\eta$, apply~\cite[Lemma~7.1]{Guo95} at $n-1$ so that its gain estimate is $\bhat_{1,n}$. With exponent $\gamma/2-1/\nu>0$, this gives $r_n=O(n)$ and $\max_{t<n}\norm{\ph_t}^2=O(n^\gamma)$. If $r_n$ is bounded, $\max_{t<n}\norm{\ph_t}^2=O(1)$.

By~\eqref{eq:ce}, $P_t\preceq P_0$ and~\cite[Corollary~3.1]{Guo95},
\[
 R_n=O\!\left(\left(1+\max_{t<n}\norm{\ph_t}^2\right)\log(e+r_n)\right) =O(n^\gamma\log n) =O(n^\delta)
\]
and $V_n=O(\log(e+r_n))$.
\end{proof}

\phantomsection\label{app:noise}\label{sec:localization-conditioning}
The results extend to random references by conditioning on the complete reference path, which fixes the reference while the noise remains random. Since $\F_n$ includes the complete reference, \eqref{eq:noise} continues to hold under this conditional law for almost every reference path, with respect to the filtration generated by the past noise~\cite[Theorems~4.1.16 and~4.1.18]{Durrett19}. Reference boundedness and~\eqref{eq:well-defined} also hold with conditional probability one, so applying the deterministic reference proofs in Sections~\ref{sec:identification}--\ref{sec:cost} and averaging over the reference establishes the almost sure conclusions. For the distributional limits, see the proof of Theorem~\ref{thm:distribution}.

\section{When the Regulator Is Well Defined}\label{app:boundary}\label{sec:zero_divisor}
A nonzero initial gain estimate ensures \eqref{eq:well-defined} if, for every $n$, the conditional law of $w_{n+1}$ given $\F_n$ has no atoms almost surely. If $\widehat b_{1,j}\ne0$ for $0\le j\le n$, use~\eqref{eq:ce} in the gain component of~\eqref{eq:rls}.
\[
 \widehat b_{1,n+1}=\widehat b_{1,n}+\kappa_n(\varepsilon_n+w_{n+1}),\qquad
 \kappa_n\triangleq\frac{e_{p+1}^TP_n\ph_n}{1+\ph_n^TP_n\ph_n}.
\]
The variables $\widehat b_{1,n}$, $\kappa_n$ and $\varepsilon_n$ are $\F_n$-measurable. If $\kappa_n=0$, the gain estimate is unchanged. Otherwise, its next value is zero only if $w_{n+1}=-\varepsilon_n-\widehat b_{1,n}/\kappa_n$, an event of conditional probability zero. Induction and a countable union prove \eqref{eq:well-defined}.

With bounded discrete noise, a gain estimate can instead vanish. Consider $y_{n+1}=4y_n+u_n+w_{n+1}$ with zero reference, zero initial histories, $P_0=I_2$ and $(\widehat a_{1,0},\widehat b_{1,0})=(-1,1)$. Then $u_0=0$ and $\ph_0=0$. Both increments in~\eqref{eq:rls}--\eqref{eq:info-update} vanish, so ${P_1=I_2}$ and $\thh_1=(1,1)^T$. Hence $y_1=w_1$, $u_1=-w_1$ and ${y_2=3w_1+w_2}$. Substituting $\ph_1=(w_1,-w_1)^T$ in \eqref{eq:rls} gives
\[
 \binom{-\widehat a_{1,2}}{\widehat b_{1,2}}
 =\binom{1}{1}+\frac{\binom{w_1}{-w_1}}{1+2w_1^2}(3w_1+w_2).
\]
For independent noise uniform on $\{-1,0,1\}$, the event $|w_1|=1,w_2=0$ has probability $2/9$. On this event, $(\widehat a_{1,2},\widehat b_{1,2})=(-2,0)$ and $2y_2+0\cdot u_2=6w_1\ne0$, so the CE control equation~\eqref{eq:ce} has no solution. The plant has $B(z)=1$ and satisfies Assumption~\ref{ass:plant}. The noise satisfies~\eqref{eq:noise} with $\sigma^2=2/3$.

\bibliographystyle{plain}
\bibliography{references}
\end{document}